\documentclass[12pt,reqno,a4paper]{amsart}
\usepackage{amsfonts,amsthm}
\usepackage{amsmath,amssymb}
\usepackage[dvipsnames]{xcolor}
\usepackage{times}
\usepackage{graphicx,enumitem}
\usepackage[margin=2.5cm]{geometry}

\usepackage{hyperref}
\hypersetup{
    colorlinks=true,
    linkcolor=blue,
    citecolor=Green,
}

\newcommand{\veps}{\varepsilon}
\newcommand{\R}{\mathbb{R}}

\newcommand{\C}{\mathbb{C}}
\newcommand{\N}{\mathbb{N}}

\newcommand{\vp}{\varphi}

\newcommand{\im}{\operatorname{Im}}
\newcommand{\re}{\operatorname{Re}}

\def\z{\zeta}
\def\el{\mathcal{L}}
\newcommand\s[2]{
\ensuremath{{_{#2}}\sqrt{#1}}}

\newtheorem{theorem}{Theorem}[section]

\theoremstyle{plain}

\newtheorem{lemma}[theorem]{Lemma}

\newtheorem{defin}[theorem]{Definition}
\newenvironment{definition}{\begin{defin}\rm}{\end{defin}}
\newtheorem{corollary}[theorem]{Corollary}

\newtheorem{proposition}[theorem]{Proposition}

{\theoremstyle{definition}
	\newtheorem{remark}[theorem]{Remark}}
\newtheorem{exa}[theorem]{Example}
\newenvironment{example}{\begin{exa}\rm}{\end{exa}}

\numberwithin{equation}{section}
\usepackage{pgf,tikz}
\usetikzlibrary{matrix}
\usetikzlibrary{arrows}
\usetikzlibrary[patterns]
\tikzset{invclip/.style={clip,insert path={{[reset cm]
				(-16383.99999pt,-16383.99999pt) rectangle (16383.99999pt,16383.99999pt)
}}}}

\begin{document}
\title[Asymptotic Integration Theory]{Asymptotic integration theory for $f'' + P(z)f = 0$}
\author{Gary~G.~Gundersen, Janne~Heittokangas and Amine~Zemirni}

\begin{abstract}

Asymptotic integration theory gives a collection of results which provide a thorough description of the asymptotic growth and zero distribution of solutions of (*) $f''+P(z)f=~0$, where $P(z)$ is a polynomial. These results have been used by several authors to find interesting properties of solutions of (*). That said, many people have remarked that the proofs and discussion concerning asymptotic integration theory that are, for example, in E.~Hille's 1969 book \emph{Lectures on Ordinary Differential Equations} are difficult to follow. The main purpose of this paper is to make this theory more understandable and accessible by giving complete explanations of the reasoning used to prove the theory and by writing full and clear statements of the results. A considerable part of the presentation and explanation of the material is different from that in Hille's book.




\end{abstract}

\maketitle


\section{Introduction and main results}


\thispagestyle{empty}

The non-trivial solutions $f$ of the differential equation
	\begin{equation}
	f^{\prime \prime }+P(z)f=0  \label{e1}
	\end{equation}
with a non-constant polynomial coefficient 
	\begin{equation}
	P(z)=p_{n}z^{n}+p_{n-1}z^{n-1}+\cdots +p_{0},\quad p_{n}\not=0,\ n\geq 1,  \label{poly}
	\end{equation}
are known to be entire, of order of growth $\rho (f)=\frac{n+2}{2}$ and of finite type. In addition, if $f_1,f_2$ are
linearly independent solutions of \eqref{e1}, then $\rho(f_1f_2)=\lambda(f_1f_2)=\max\left\{\lambda(f_1),\lambda(f_2)\right\}=\frac{n+2}{2}$, see \cite{laine}. Here $\lambda(g)$ denotes the exponent
of convergence of the zero sequence of $g$. It is easy to see that the zeros of $f$ are simple, but
it is non-trivial to find the location of the zeros. Indeed, all but finitely many zeros of $f$ lie in the $n+2$ sectors
	\begin{equation}\label{sectors+rays}
	W_{j}\left( \varepsilon \right) =\left\{ z:|\arg (z)-\theta _{j}|<\varepsilon\right\},
	\quad \theta _{j}=\frac{2\pi j-\arg(p_n)}{n+2},
	\end{equation}
where $-\pi<\arg(p_n)\le\pi$, $\varepsilon >0$ is arbitrarily small and $j=0,\ldots,n+1$. 
 Chapter~$ 7.4 $ in Hille's book \cite{hille} is the general citation to this property, although the method of proof originates from his earlier papers \cite{hille-1,hille3}. Sometimes \cite[Ch.~5.6]{hille2} is also cited. 
%
%
In this connection, see 
the papers by Fuchs \cite{Fuchs}, Nevanlinna \cite{Nevanlinna} and Wittich \cite{Wittich}.

Hille remarked in \cite[p.~342]{hille} that 
if a solution $f\not\equiv 0$ of \eqref{e1} has
infinitely many zeros $\{z_k\}$ in a sector $W_j(\varepsilon)$, then these zeros approach the
\emph{critical ray} $\arg(z)=\theta_j$. This remark is not true in general due to a
counterexample by Bank \cite{bank0}. However, it was proved by Bank \cite{bank0} (using Strodt's theory) and by Hellerstein-Rossi \cite{hell-rossi} (using Hille's approach), that the zeros approach the translate $\arg (z+c)=\theta_j$ of the critical ray, where $c=p_{n-1}/(np_n)$. Hence, the remark of Hille is true exactly when $p_{n-1}=0$. If this sequence of zeros $\{z_k\}$ in $W_j(\varepsilon)$ has a non-decreasing modulus, then the distance of the zeros $z_k$ from the ray $\arg (z+c)=\theta_j$ is $O(r_ke_n(r_k))$, as  $r_k=|z_k|\to\infty$, where
	\begin{equation}\label{er}
		e_n(r)= 
		\left\{
			\begin{array}{ll}
			r^{-2},\ & n>2,\\
			r^{-2} \log r,\ & n=2,\\
			r^{-3/2},\ & n=1.
			\end{array}
		\right.
		\end{equation}
This was proved in \cite{bank0, bank1,hell-rossi}.

Other than the above incorrect statement which was corrected by Bank and Hellerstein-Rossi, the rest of the asympototic theory of integration in \cite[Ch.~7.4]{hille} appears correct. The ideas in \cite[Ch.~7.4]{hille} are profound, the methods are creative,
and the results of the theory have been very useful in numerous applications to differential equations \cite{BT, gund, gund2, hell-rossi0, hell-rossi, langley, LH, long, long2, long3, shin, WLHQ}. At the same time, many readers find the material in \cite[Ch.~7.4]{hille} difficult to follow because (a)~general statements of some of the basic results and consequences of the theory are not clearly and fully stated, and (b) details that would justify several steps in the proofs are omitted. These issues can make the reading of \cite[Ch.~7.4]{hille} laborious.


Accordingly, our purposes are as follows:
	\begin{itemize}
	\item[(I)] We write clearly stated results together with rigorous proofs on the growth and zero distribution of solutions of \eqref{e1} that are obtained from Hille's approach. For the zero distribution, we take into
	account the above convergence rate $O(r_ke_n(r_k))$ of the zeros $\{z_k\}$ toward the \emph{critical translates} $\arg (z+c)=\theta_j$, where $j=0,\ldots,n+1$.
	\item[(II)] To obtain the results in (I), we need Liouville's transformation and Hille's theory on asymptotic integration, which
	we discuss in detail in such a way that the
	explanations will be more accessible to non-expert readers who have some background in complex analysis. Some of the 
	presentations and explanations will be in revised frameworks and will be different from the material in \cite[Ch.~7.4]{hille}. 
	\end{itemize}

In the literature, the following two simplified facts are often referenced to \cite[Ch.~7.4]{hille}:  
	\begin{itemize}
	\item[(A)] On all the rays in an open sector	 $S(\theta_{j-1},\theta_{j})$ determined by any two consecutive critical rays, 
	each solution $f\not\equiv 0$ either blows up on each ray or decays to zero exponentially on each ray.
	By this we mean, respectively, that
		$$
		\liminf_{r\to \infty} r^{-(n+2)/2}\log |f(re^{i\theta})|>0
		\quad\textnormal{or}\quad
		\liminf_{r\to \infty} r^{-(n+2)/2}\log |f(re^{i\theta})|^{-1}>0,
		$$
	for $\theta\in (\theta_{j-1},\theta_{j})$. 
	\item[(B)] If a sector $W_{j}\left( \varepsilon \right) =\left\{ z:|\arg (z)-\theta _{j}|<\varepsilon\right\}$ contains infinitely many zeros, then the number of zeros
	in $W_{j}\left( \varepsilon \right)$ is asymptotically comparable to $r^{(n+2)/2}$.
	\end{itemize}

In Theorems \ref{theo-growth} and \ref{theo-zeros} below we write more precise statements of these results. 
For this purpose, we will use the notation 
	\begin{equation}\label{consts}
	c=\frac{p_{n-1}}{np_n},\quad q=\frac{n+2}{2},\quad 
	d= \frac{(p_n)^{1/2}}{q},
	\end{equation}
where $(p_n)^{1/2}= \sqrt{|p_n|}\exp\left(i \frac{\arg(p_n)}{2}\right)$ with $-\pi<\arg (p_n)\le \pi$. 
The first result concerns the exponential growth and decay of solutions to make (A)  precise.

\begin{theorem}\label{theo-growth}
Let $f$ be a non-trivial solution of \eqref{e1}. Then the following statements hold:
\begin{itemize}
\item[\textnormal{(a)}] In any given open sector $ S $ between consecutive critical rays, $ f $  either (i) blows up exponentially on all the rays $ \arg (z)= \theta $ in $ S $, or (ii) decays exponentially to zero on all the rays $ \arg (z)= \theta $ in $ S $.
Specifically, on all the rays $ \arg (z)= \theta $ in $ S $, $ f $ is asymptotically comparable to either
	$$
	E_1(z)= \exp \left\{  id z^q (1+o(1))\right\}\quad \text{or} \quad E_2(z)= \exp \left\{ - id z^q (1+o(1))\right\}.
	$$
Moreover, in each such sector $ S $, there exist solutions of \eqref{e1} of both types (i)  and (ii).

\item[\textnormal{(b)}] In any two adjacent sectors $S(\theta_{j-1},\theta_j)$ and $S(\theta_j,\theta_{j+1})$ that border one common critical ray $\arg (z)=\theta_j$, there cannot exist a ray in $S(\theta_{j-1},\theta_j)$ and another ray in $S(\theta_j,\theta_{j+1})$ such that $ f $ decays exponentially to zero on both rays.
Here, $\theta_{-1}=\theta_{n+1}$.
\end{itemize}
\end{theorem}


In Theorem~\ref{theo-growth}(a), any branch cut outside the sector $S$ and any branch for the square roots in the
expressions $E_1(z)$ and $E_2(z)$  can be chosen as long as they are the same for both $E_1(z)$ and $E_2(z)$. For some choices, the roles of $E_1(z)$ and $E_2(z)$ will be interchanged. This will be explained in Appendix~\ref{appendix}.


Instead of the classical $\varepsilon$-sectors in \eqref{sectors+rays}, when addressing (I), we consider the domains 
		$$
		\Lambda_j = \{z=re^{i\theta}: r>R,\; |\theta-\theta_j|<C e_n(r)\}
		$$
and their translates
		\begin{equation}\label{CS}
		\Lambda_{j,c}=\{z: z+c \in \Lambda_j\}.
		\end{equation}
Here $R>0$ is large enough, $C=C(n,R)>0$, and $c$ is defined in \eqref{consts}. 
The second result concerns the distribution of zeros of solutions to make (B)  precise. 

\begin{theorem}\label{theo-zeros}
Let $f$ be a non-trivial solution of \eqref{e1}. Then all but at most finitely many zeros of $f$ lie
in the union 	
		$$
		\bigcup_{j=0}^{n+1} \Lambda_{j,c}.
		$$
If $f$ has infinitely many zeros in $\Lambda_{j,c}$,  then
	\begin{eqnarray}
	n(r,\Lambda_{j,c} ,1/f)&=&\frac{\sqrt{|p_{n}|}}{q\pi }r^{q}\left( 1+o(1)\right),
	\quad r\rightarrow \infty,\label{n}\\
	N(r,\Lambda_{j,c} ,1/f)&=&\frac{\sqrt{|p_{n}|}}{q^2 \pi }r^{q}\left( 1+o(1)\right),
	\quad r\rightarrow \infty,\label{N}
	\end{eqnarray} 
where the counting function $n(r,\Lambda_{j,c} ,1/f)$ refers to only
those zeros in $\Lambda_{j,c}$ satisfying $ |z|\le r $ and $N(r,\Lambda_{j,c} ,1/f)$ is the corresponding integrated counting function.
\end{theorem}

From this result it is clear that the zeros of $f$ approach the critical translates 
	$$
	\arg(z+c)=\theta_j,\quad j=0,\ldots,n+1,
	$$ 
emanating from the point $-c$, see Figure~\ref{fig00}. It is easy to see that the sector $W_j(\veps)$ enclosing the critical ray $\arg (z)=\theta_j$ contains the essential part of $\Lambda_{j,c}$, no matter how large $|c|$ is.
Therefore, it follows that all but at most finitely many zeros of $f$ are located in the union of the sectors $W_j(\veps)$, as is known previously.

\begin{figure}[h!]
	\centering
	\definecolor{ffqqqq}{rgb}{1,0,0}
	\definecolor{qqqqff}{rgb}{0,0,1}
	\begin{tikzpicture}[line cap=round,line join=round,>=triangle 45,x=1.0cm,y=1.0cm]
	\clip(8.18,0.3) rectangle (20.14,10.52);
	
	\begin{scope}
	\begin{pgfinterruptboundingbox} 
	\path[invclip](12.73,4.48) circle (2.8cm);
	\end{pgfinterruptboundingbox}
	\draw[fill=black,fill opacity=0.1, draw opacity=0] (20.09,2.12) -- (20.04,2.15) -- (19.95,2.2) -- (19.86,2.25) -- (19.77,2.29) -- (19.69,2.34) -- (19.61,2.38) -- (19.53,2.42) -- (19.46,2.46) -- (19.39,2.5) -- (19.32,2.53) -- (19.26,2.57) -- (19.19,2.6) -- (19.13,2.63) -- (19.08,2.66) -- (19.02,2.69) -- (18.96,2.72) -- (18.91,2.75) -- (18.86,2.78) -- (18.76,2.83) -- (18.67,2.88) -- (18.58,2.93) -- (18.5,2.98) -- (18.43,3.02) -- (18.35,3.06) -- (18.28,3.1) -- (18.22,3.14) -- (18.15,3.17) -- (18.09,3.21) -- (18.03,3.24) -- (17.98,3.27) -- (17.92,3.3) -- (17.87,3.33) -- (17.82,3.36) -- (17.73,3.41) -- (17.64,3.46) -- (17.56,3.51) -- (17.49,3.56) -- (17.42,3.6) -- (17.35,3.64) -- (17.29,3.67) -- (17.23,3.71) -- (17.18,3.74) -- (17.12,3.78) -- (17.07,3.81) -- (16.98,3.87) -- (16.9,3.92) -- (16.82,3.97) -- (16.75,4.02) -- (16.68,4.06) -- (16.62,4.1) -- (16.56,4.14) -- (16.51,4.18) -- (16.46,4.21) -- (16.37,4.28) -- (16.31,4.32) -- (16.25,4.37) -- (16.15,4.44) -- (16.07,4.51) -- (15.99,4.58) -- (15.92,4.64) -- (15.86,4.69) -- (15.81,4.75) -- (15.76,4.8) -- (15.72,4.84) -- (15.68,4.89) -- (15.64,4.93) -- (15.61,4.97) -- (15.58,5.01) -- (15.55,5.05) -- (15.53,5.09) -- (15.5,5.13) -- (15.48,5.16) -- (15.46,5.2) -- (15.44,5.23) -- (15.43,5.26) -- (15.41,5.3) -- (15.39,5.33) -- (15.38,5.36) -- (15.37,5.39) -- (15.36,5.42) -- (15.35,5.45) -- (15.34,5.48) -- (15.33,5.51) -- (15.32,5.53)
	
	-- (15.32,5.53) -- (15.32,5.56) -- (15.31,5.59) -- (15.3,5.62) -- (15.3,5.65) -- (15.29,5.68) -- (15.29,5.71) -- (15.28,5.74) -- (15.28,5.77) -- (15.28,5.8) -- (15.28,5.83) -- (15.27,5.86) -- (15.27,5.89) -- (15.27,5.92) -- (15.27,5.95) -- (15.27,5.98) -- (15.27,6.01) -- (15.27,6.05) -- (15.27,6.08) -- (15.28,6.11) -- (15.28,6.15) -- (15.28,6.18) -- (15.29,6.22) -- (15.29,6.26) -- (15.3,6.3) -- (15.3,6.34) -- (15.31,6.38) -- (15.32,6.42) -- (15.33,6.47) -- (15.34,6.51) -- (15.35,6.56) -- (15.36,6.62) -- (15.37,6.67) -- (15.39,6.73) -- (15.4,6.79) -- (15.42,6.86) -- (15.44,6.93) -- (15.46,7.01) -- (15.49,7.09) -- (15.52,7.18) -- (15.55,7.28) -- (15.57,7.33) -- (15.59,7.39) -- (15.61,7.45) -- (15.63,7.51) -- (15.66,7.58) -- (15.69,7.65) -- (15.72,7.73) -- (15.75,7.81) -- (15.78,7.9) -- (15.82,7.99) -- (15.84,8.04) -- (15.86,8.1) -- (15.89,8.15) -- (15.91,8.21) -- (15.93,8.27) -- (15.96,8.34) -- (15.99,8.41) -- (16.02,8.48) -- (16.05,8.55) -- (16.09,8.64) -- (16.12,8.72) -- (16.16,8.81) -- (16.21,8.91) -- (16.23,8.96) -- (16.25,9.02) -- (16.28,9.07) -- (16.3,9.13) -- (16.33,9.19) -- (16.36,9.25) -- (16.39,9.32) -- (16.42,9.38) -- (16.45,9.45) -- (16.48,9.53) -- (16.52,9.61) -- (16.56,9.69) -- (16.59,9.77) -- (16.64,9.86) -- (16.68,9.96) -- (16.7,10.01) -- (16.73,10.06) -- (16.75,10.11) -- (16.78,10.17) -- (16.8,10.22) -- (16.83,10.28) -- (16.86,10.34) -- (16.89,10.4) -- (16.91,10.47) -- (16.95,10.53) 
	
	--(16.54,10.46) -- (16.51,10.39) -- (16.47,10.32) -- (16.43,10.26) -- (16.4,10.2) -- (16.37,10.14) -- (16.34,10.08) -- (16.31,10.02) -- (16.28,9.97) -- (16.25,9.91) -- (16.22,9.86) -- (16.17,9.76) -- (16.12,9.67) -- (16.07,9.59) -- (16.02,9.5) -- (15.98,9.43) -- (15.94,9.35) -- (15.9,9.28) -- (15.86,9.22) -- (15.83,9.15) -- (15.79,9.09) -- (15.76,9.03) -- (15.73,8.98) -- (15.7,8.92) -- (15.67,8.87) -- (15.64,8.82) -- (15.59,8.73) -- (15.54,8.65) -- (15.49,8.57) -- (15.45,8.49) -- (15.4,8.42) -- (15.36,8.35) -- (15.33,8.29) -- (15.29,8.23) -- (15.26,8.18) -- (15.22,8.12) -- (15.19,8.07) -- (15.13,7.98) -- (15.08,7.9) -- (15.03,7.82) -- (14.98,7.75) -- (14.94,7.68) -- (14.9,7.62) -- (14.86,7.56) -- (14.82,7.51) -- (14.79,7.46) -- (14.72,7.37) -- (14.68,7.31) -- (14.63,7.25) -- (14.56,7.15) -- (14.49,7.07) -- (14.42,6.99) -- (14.36,6.92) -- (14.31,6.86) -- (14.25,6.81) -- (14.2,6.76) -- (14.16,6.72) -- (14.11,6.68) -- (14.07,6.64) -- (14.03,6.61) -- (13.99,6.58) -- (13.95,6.55) -- (13.91,6.53) -- (13.87,6.5) -- (13.84,6.48) -- (13.8,6.46) -- (13.77,6.44) -- (13.74,6.43) -- (13.7,6.41) -- (13.67,6.39) -- (13.64,6.38) -- (13.61,6.37) -- (13.58,6.36) -- (13.55,6.35) -- (13.52,6.34) -- (13.49,6.33) -- (13.47,6.32)
	
	-- (13.47,6.32) -- (13.44,6.32) -- (13.41,6.31) -- (13.38,6.3) -- (13.35,6.3) -- (13.32,6.29) -- (13.29,6.29) -- (13.26,6.28) -- (13.23,6.28) -- (13.2,6.28) -- (13.17,6.28) -- (13.14,6.27) -- (13.11,6.27) -- (13.08,6.27) -- (13.05,6.27) -- (13.02,6.27) -- (12.99,6.27) -- (12.95,6.27) -- (12.92,6.27) -- (12.89,6.28) -- (12.85,6.28) -- (12.82,6.28) -- (12.78,6.29) -- (12.74,6.29) -- (12.7,6.3) -- (12.66,6.3) -- (12.62,6.31) -- (12.58,6.32) -- (12.53,6.33) -- (12.49,6.34) -- (12.44,6.35) -- (12.38,6.36) -- (12.33,6.37) -- (12.27,6.39) -- (12.21,6.4) -- (12.14,6.42) -- (12.07,6.44) -- (11.99,6.46) -- (11.91,6.49) -- (11.82,6.52) -- (11.72,6.55) -- (11.67,6.57) -- (11.61,6.59) -- (11.55,6.61) -- (11.49,6.63) -- (11.42,6.66) -- (11.35,6.69) -- (11.27,6.72) -- (11.19,6.75) -- (11.1,6.78) -- (11.01,6.82) -- (10.96,6.84) -- (10.9,6.86) -- (10.85,6.89) -- (10.79,6.91) -- (10.73,6.93) -- (10.66,6.96) -- (10.59,6.99) -- (10.52,7.02) -- (10.45,7.05) -- (10.36,7.09) -- (10.28,7.12) -- (10.19,7.16) -- (10.09,7.21) -- (10.04,7.23) -- (9.98,7.25) -- (9.93,7.28) -- (9.87,7.3) -- (9.81,7.33) -- (9.75,7.36) -- (9.68,7.39) -- (9.62,7.42) -- (9.55,7.45) -- (9.47,7.48) -- (9.39,7.52) -- (9.31,7.56) -- (9.23,7.59) -- (9.14,7.64) -- (9.04,7.68) -- (8.99,7.7) -- (8.94,7.73) -- (8.89,7.75) -- (8.83,7.78) -- (8.78,7.8) -- (8.72,7.83) -- (8.66,7.86) -- (8.6,7.89) -- (8.53,7.91) -- (8.47,7.95) -- (8.4,7.98) -- (8.33,8.01) -- (8.26,8.04) -- (8.18,8.08) -- (8.1,8.12)

	-- (8.23,7.71) -- (8.31,7.66) -- (8.39,7.62) -- (8.47,7.58) -- (8.54,7.54) -- (8.61,7.5) -- (8.68,7.47) -- (8.74,7.43) -- (8.81,7.4) -- (8.87,7.37) -- (8.92,7.34) -- (8.98,7.31) -- (9.04,7.28) -- (9.09,7.25) -- (9.14,7.22) -- (9.24,7.17) -- (9.33,7.12) -- (9.42,7.07) -- (9.5,7.02) -- (9.57,6.98) -- (9.65,6.94) -- (9.72,6.9) -- (9.78,6.86) -- (9.85,6.83) -- (9.91,6.79) -- (9.97,6.76) -- (10.02,6.73) -- (10.08,6.7) -- (10.13,6.67) -- (10.18,6.64) -- (10.27,6.59) -- (10.36,6.54) -- (10.44,6.49) -- (10.51,6.44) -- (10.58,6.4) -- (10.65,6.36) -- (10.71,6.33) -- (10.77,6.29) -- (10.82,6.26) -- (10.88,6.22) -- (10.93,6.19) -- (11.02,6.13) -- (11.1,6.08) -- (11.18,6.03) -- (11.25,5.98) -- (11.32,5.94) -- (11.38,5.9) -- (11.44,5.86) -- (11.49,5.82) -- (11.54,5.79) -- (11.63,5.72) -- (11.69,5.68) -- (11.75,5.63) -- (11.85,5.56) -- (11.93,5.49) -- (12.01,5.42) -- (12.08,5.36) -- (12.14,5.31) -- (12.19,5.25) -- (12.24,5.2) -- (12.28,5.16) -- (12.32,5.11) -- (12.36,5.07) -- (12.39,5.03) -- (12.42,4.99) -- (12.45,4.95) -- (12.47,4.91) -- (12.5,4.87) -- (12.52,4.84) -- (12.54,4.8) -- (12.56,4.77) -- (12.57,4.74) -- (12.59,4.7) -- (12.61,4.67) -- (12.62,4.64) -- (12.63,4.61) -- (12.64,4.58) -- (12.65,4.55) -- (12.66,4.52) -- (12.67,4.49) -- (12.68,4.47)
	
	-- (12.68,4.47) -- (12.68,4.44) -- (12.69,4.41) -- (12.7,4.38) -- (12.7,4.35) -- (12.71,4.32) -- (12.71,4.29) -- (12.72,4.26) -- (12.72,4.23) -- (12.72,4.2) -- (12.72,4.17) -- (12.73,4.14) -- (12.73,4.11) -- (12.73,4.08) -- (12.73,4.05) -- (12.73,4.02) -- (12.73,3.99) -- (12.73,3.95) -- (12.73,3.92) -- (12.72,3.89) -- (12.72,3.85) -- (12.72,3.82) -- (12.71,3.78) -- (12.71,3.74) -- (12.7,3.7) -- (12.7,3.66) -- (12.69,3.62) -- (12.68,3.58) -- (12.67,3.53) -- (12.66,3.49) -- (12.65,3.44) -- (12.64,3.38) -- (12.63,3.33) -- (12.61,3.27) -- (12.6,3.21) -- (12.58,3.14) -- (12.56,3.07) -- (12.54,2.99) -- (12.51,2.91) -- (12.48,2.82) -- (12.45,2.72) -- (12.43,2.67) -- (12.41,2.61) -- (12.39,2.55) -- (12.37,2.49) -- (12.34,2.42) -- (12.31,2.35) -- (12.28,2.27) -- (12.25,2.19) -- (12.22,2.1) -- (12.18,2.01) -- (12.16,1.96) -- (12.14,1.9) -- (12.11,1.85) -- (12.09,1.79) -- (12.07,1.73) -- (12.04,1.66) -- (12.01,1.59) -- (11.98,1.52) -- (11.95,1.45) -- (11.91,1.36) -- (11.88,1.28) -- (11.84,1.19) -- (11.79,1.09) -- (11.77,1.04) -- (11.75,0.98) -- (11.72,0.93) -- (11.7,0.87) -- (11.67,0.81) -- (11.64,0.75) -- (11.61,0.68) -- (11.58,0.62) -- (11.55,0.55) -- (11.52,0.47) -- (11.48,0.39) -- (11.44,0.31) -- (11.41,0.23)
	
	-- (11.88,0.33) -- (11.93,0.41) -- (11.98,0.5) -- (12.02,0.57) -- (12.06,0.65) -- (12.1,0.72) -- (12.14,0.78) -- (12.17,0.85) -- (12.21,0.91) -- (12.24,0.97) -- (12.27,1.02) -- (12.3,1.08) -- (12.33,1.13) -- (12.36,1.18) -- (12.41,1.27) -- (12.46,1.35) -- (12.51,1.43) -- (12.55,1.51) -- (12.6,1.58) -- (12.64,1.65) -- (12.67,1.71) -- (12.71,1.77) -- (12.74,1.82) -- (12.78,1.88) -- (12.81,1.93) -- (12.87,2.02) -- (12.92,2.1) -- (12.97,2.18) -- (13.02,2.25) -- (13.06,2.32) -- (13.1,2.38) -- (13.14,2.44) -- (13.18,2.49) -- (13.21,2.54) -- (13.28,2.63) -- (13.32,2.69) -- (13.37,2.75) -- (13.44,2.85) -- (13.51,2.93) -- (13.58,3.01) -- (13.64,3.08) -- (13.69,3.14) -- (13.75,3.19) -- (13.8,3.24) -- (13.84,3.28) -- (13.89,3.32) -- (13.93,3.36) -- (13.97,3.39) -- (14.01,3.42) -- (14.05,3.45) -- (14.09,3.47) -- (14.13,3.5) -- (14.16,3.52) -- (14.2,3.54) -- (14.23,3.56) -- (14.26,3.57) -- (14.3,3.59) -- (14.33,3.61) -- (14.36,3.62) -- (14.39,3.63) -- (14.42,3.64) -- (14.45,3.65) -- (14.48,3.66) -- (14.51,3.67) -- (14.53,3.68)
	-- (14.53,3.68) -- (14.56,3.68) -- (14.59,3.69) -- (14.62,3.7) -- (14.65,3.7) -- (14.68,3.71) -- (14.71,3.71) -- (14.74,3.72) -- (14.77,3.72) -- (14.8,3.72) -- (14.83,3.72) -- (14.86,3.73) -- (14.89,3.73) -- (14.92,3.73) -- (14.95,3.73) -- (14.98,3.73) -- (15.01,3.73) -- (15.05,3.73) -- (15.08,3.73) -- (15.11,3.72) -- (15.15,3.72) -- (15.18,3.72) -- (15.22,3.71) -- (15.26,3.71) -- (15.3,3.7) -- (15.34,3.7) -- (15.38,3.69) -- (15.42,3.68) -- (15.47,3.67) -- (15.51,3.66) -- (15.56,3.65) -- (15.62,3.64) -- (15.67,3.63) -- (15.73,3.61) -- (15.79,3.6) -- (15.86,3.58) -- (15.93,3.56) -- (16.01,3.54) -- (16.09,3.51) -- (16.18,3.48) -- (16.28,3.45) -- (16.33,3.43) -- (16.39,3.41) -- (16.45,3.39) -- (16.51,3.37) -- (16.58,3.34) -- (16.65,3.31) -- (16.73,3.28) -- (16.81,3.25) -- (16.9,3.22) -- (16.99,3.18) -- (17.04,3.16) -- (17.1,3.14) -- (17.15,3.11) -- (17.21,3.09) -- (17.27,3.07) -- (17.34,3.04) -- (17.41,3.01) -- (17.48,2.98) -- (17.55,2.95) -- (17.64,2.91) -- (17.72,2.88) -- (17.81,2.84) -- (17.91,2.79) -- (17.96,2.77) -- (18.02,2.75) -- (18.07,2.72) -- (18.13,2.7) -- (18.19,2.67) -- (18.25,2.64) -- (18.32,2.61) -- (18.38,2.58) -- (18.45,2.55) -- (18.53,2.52) -- (18.61,2.48) -- (18.69,2.44) -- (18.77,2.41) -- (18.86,2.36) -- (18.96,2.32) -- (19.01,2.3) -- (19.06,2.27) -- (19.11,2.25) -- (19.17,2.22) -- (19.22,2.2) -- (19.28,2.17) -- (19.34,2.14) -- (19.4,2.11) -- (19.47,2.09) -- (19.53,2.05) -- (19.6,2.02) -- (19.67,1.99) -- (19.74,1.96) -- (19.82,1.92) -- (19.9,1.88) -- (19.98,1.84) -- (20.06,1.8) -- (20.15,1.76) -- cycle;
	\end{scope}

	\draw [line width=0.4pt,fill=black,fill opacity=0.1, draw opacity=0] (12.73,4.48) circle (2.8cm);
	
	\draw [dash pattern=on 3pt off 3pt,color=qqqqff,domain=8.18:20.14] plot(\x,{(--48-2*\x)/4});
	\draw [dash pattern=on 3pt off 3pt,color=qqqqff,domain=8.18:20.14] plot(\x,{(-46--4*\x)/2});
	
	\draw [dash pattern=on 2pt off 2pt,domain=8.18:20.14] plot(\x,{(--80.09-3.73*\x)/7.28});
	\draw [dash pattern=on 2pt off 2pt,domain=8.18:20.14] plot(\x,{(-75.99--7.28*\x)/3.73});

	\draw [->,line width=0.4pt] (12.73,4.48) -- (19.7,4.4);
	\draw [->,line width=0.4pt] (12.73,4.48) -- (12.7,10.4);

	\begin{scriptsize}
	\fill [color=qqqqff] (14,5) circle (1.5pt);
	\end{scriptsize}

	\draw (14.1,5.4) node[anchor=north west] {$-c$};
	\draw [-latex] (17.3,9.7) -- (15.4,9.7);
	\draw (17.25,10.04) node[anchor=north west] {\small critical ray};
	\draw [-latex] (17.3,8.8) -- (15.9,8.8);
	\draw (17.25,9.15) node[anchor=north west] {\small critical translate};
	
	\draw (14.1,2.25) node[anchor=north west] {$R$};
	
	\end{tikzpicture}
	\caption{Domains $\Lambda_{j,c}$ around the critical translates in the case $ n=2 $. All the zeros of $ f $  lie in the shaded area when $ R>0 $ is large enough.}
	\label{fig00}
\end{figure}
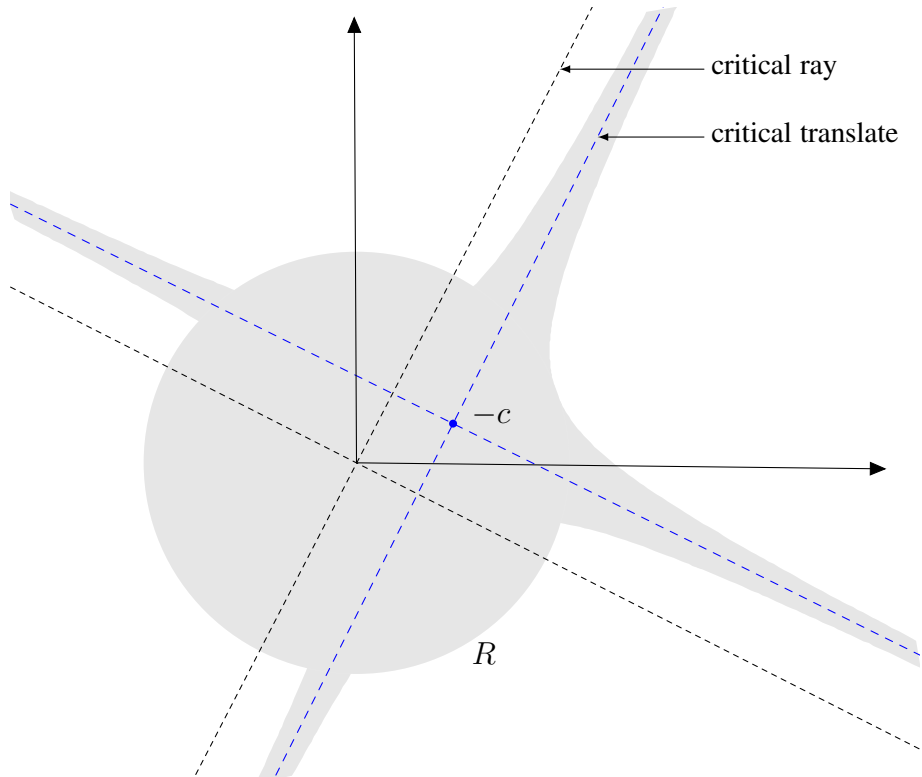

If $f$ has only finitely many zeros in $W_j(\veps)$, then the critical ray $\arg(z)=\theta_j$ is called a {\it shortage ray} of $f$, otherwise it is a {\it non-shortage ray} of $f$.  Regarding the concept of {\it shortage} in this situation, see \cite{gund}.
The third result  reveals the interplay between exponential growth/decay and non-shortage/shortage rays.

\begin{theorem}\label{theo-growth-zero}
Let $f$ be a non-trivial solution of \eqref{e1}. Then the following statements hold:
\begin{itemize}
\item[\textnormal{(a)}] If $f$ blows up exponentially on each ray in two adjacent sectors $S(\theta_{j-1},\theta_j)$ and $S(\theta_j,\theta_{j+1})$ that border a common critical ray $\arg(z)=\theta_j$, then the critical ray $\arg (z)=\theta_j$ is non-shortage.
\\[-20pt]

\item[\textnormal{(b)}] If $f$ decays to zero exponentially on all the rays in a sector $S(\theta_j,\theta_{j+1})$, then both critical rays $\arg(z)=\theta_j,\theta_{j+1}$ are shortage. \\[-20pt]
\end{itemize}
\end{theorem}


For the relationship between the number of shortage rays and the Nevanlinna functions $N(r,1/f)$, $T(r,f)$, $\delta(0,f)$ 
of a solution $f$ of \eqref{e1}, see \cite{gund,hell-rossi0}.
  
The remainder of this paper is organized as follows. To complete (I), the proofs of  Theorems~\ref{theo-growth}, \ref{theo-zeros} and \ref{theo-growth-zero} are given in Sections~\ref{proof-of-thm1}--\ref{proof-of-thm3}, respectively.  Regarding (II), revised discussions on Liouville's transformation and on asymptotic integration are given in Sections~\ref{Liouville-sec} and
\ref{asymp-int}, respectively. The purpose of Sections~\ref{Liouville-sec} and \ref{asymp-int} is to make a connection between equation \eqref{e1} and the sine equation 
	\begin{equation}\label{sine-equation}
	w''+w=0,
	\end{equation}
so that the well-known growth and zero distribution properties of the three types of solutions
$e^{iz}, e^{-iz}, \sin(z-z_0)$ of \eqref{sine-equation} can be used to prove the
analogous properties of the solutions of \eqref{e1} in Theorems~\ref{theo-growth}--\ref{theo-growth-zero}.  
To make this connection, we first use a routine change of variable to transform \eqref{e1} into a more convenient equation $ g'' + Q(z)g =0$, where the polynomial $ Q(z) $ has been normalized to the form 
	\begin{equation}\label{Q}
	Q(z) = \left\{
	\begin{array}{ll}
	z, & n=1,\\
	z^n + a_{n-2}z^{n-2} + \cdots + a_0,\; & n \geq 2.
	\end{array}
	\right.
	\end{equation}
Then the process of making  the connection consists of the two main steps  in Figure~\ref{process}, 
where $w''+[1-T(\zeta)]w=0$ is a perturbed sine equation with $ T(\zeta) = O(\zeta^{-2}) $ as $ \zeta \to \infty $.
	\begin{figure}[h!]
	\begin{tikzpicture}
	\matrix (m) [matrix of math nodes,row sep=3em,column sep=4em,minimum width=2em]
	{
	f''+P(z)f=0&	g'' + Q(z)g =0\\ 
	 w'' + w=0 & w''+[1-T(\zeta)]w=0\\ 
	};
	\path[-stealth]
	(m-1-1)  edge [<->] node [above] {\small  change of } node[below] {\small variable} (m-1-2)
	(m-1-2)  edge [<->] node [right] {\small  Liouville's transformation}  (m-2-2)
	(m-2-2) edge [<->] node [above] {\small  asymptotic} node [below] {\small integration} (m-2-1);
	\end{tikzpicture}
	\caption{The connection between equations.}
	\label{process}
\end{figure}



\section{Liouville's transformation}\label{Liouville-sec}


Liouville's transformation can be applied to more general equations \cite[p.~340]{hille}, but it is enough for our purposes to apply it to \eqref{e1}, which will somewhat simplify the presentation and discussion. This will transform equation \eqref{e1} into a perturbed sine equation, and asymptotic integration is then used on the perturbed sine equation to obtain the proofs of Theorems~\ref{theo-growth}--\ref{theo-growth-zero}.

\subsection{\sc Preparations}\label{preli-sec}

For convenience, in the rest of the paper we will assume that the polynomial $P(z)$ in \eqref{poly} is normalized, that is, we consider the equation
	\begin{equation} \label{s-e1}
	g'' + Q(z)g = 0
	\end{equation}
with a polynomial coefficient $ Q(z) $ defined as in \eqref{Q}. 
There is no loss of generality in doing this because results about solutions of \eqref{s-e1} can be transformed into results about solutions of \eqref{e1} by observing that if $f$ is a solution of \eqref{e1}, then from $ c=\frac{p_{n-1}}{np_n} $ in \eqref{consts}, it can be verified that 
	\begin{equation}\label{g}
	g(z) = f(\mu z-c) 
	\end{equation}
will be a solution of \eqref{s-e1} when $\mu$ is a constant satisfying $\mu^{n+2} = p_n^{-1}$. 
It turns out to be useful to re-write $Q(z)$ in \eqref{Q} as
	\begin{equation}\label{poly_Q}
	Q(z) = z^n (1+\ell(z)), \quad n\ge 1,
	\end{equation}
where $\ell(z)\equiv 0$ if $n=1$ and 
	$$
	\ell(z) =  \frac{a_{n-2}}{z^2} + \cdots+ \frac{a_0}{z^n}, \quad n\ge 2.
	$$

From \eqref{sectors+rays}, the critical rays of \eqref{s-e1} are $\arg(z)=\psi_j$, where
	\begin{equation}\label{s-critical rays}
	\psi_j = \frac{2 \pi j}{n+2}, \quad j = 0, 1, \ldots , n + 1. 
	\end{equation}
For $ j=0, \ldots, n+1 $, let $ G_j(R) $ denote the domain 
	\begin{equation}\label{G_j}
	G_j(R) = \left\{z\in\C: |z|>R,\, \psi_{j-1}< \arg (z) < \psi_{j+1} \right\},
	\end{equation}
where $ \psi_{-1}=\psi_{n+1}-2\pi $ and $ \psi_{n+2}=\psi_0+2\pi $, and $ R>0 $ is large enough to satisfy the indicated conditions below.
Moreover, for a non-trivial solution $g$ of \eqref{s-e1}, let $ g_j $ denote the restriction of $ g $ 
to $ G_j(R)$.

According to Hille (\cite[p.~340]{hille}, \cite[p.~179]{hille2}), the transformation
	\begin{equation}\label{LT0}
	\left\{ \begin{array}{l}
	\zeta =L(z)=\displaystyle\int_{z_{0}}^{z}Q(\xi)^{1/2}\,d\xi,\\ 
	w(\zeta) =Q(\zeta)^{1/4} g(\zeta),
	\end{array}\right.
	\end{equation}
was first introduced by Liouville in 1837. This transformation was originally used on the positive real line, where there are no questions about branches. However, for the complex variables in \eqref{LT0}, 
questions about branches have to be addressed, and we think it is easier to deal with these
questions if we use \eqref{poly_Q} to re-write \eqref{LT0} in the following manner.
For $ z\in G_j(R) $, we write Liouville's transformation \eqref{LT0} as
	\begin{equation}\label{LT}
	\left\{ 
	\begin{array}{l}
	\zeta =L_j( z) =\displaystyle\int_{z_{0}}^{z}\xi ^{n/2}\left( 1+\ell
	\left( \xi \right) \right) ^{1/2}d\xi , \\ 
	w_j( \zeta) =z^{{n}/{4}}\left( 1+\ell \left( z\right) \right)^{{1}/{4}} g_j(z).
	\end{array}
	\right.   
	\end{equation}
We choose $ z_0=2R e^{i\psi_j} $. Then the points $ z_0,z\in G_j(R) $ can be 
	connected with at most two line segments lying in $ G_j(R) $. For the path of integration in
	\eqref{LT}, we choose this polygonal path.
	Further, we require the constant $ R $ in \eqref{G_j} to be large enough such that
			\begin{equation}\label{RR}
			R\ge \max\left\{1,\sqrt{(n-1) M_0}\right\},
			\end{equation}
		where $M_0=0$ if $n=1$ and
			\begin{equation}\label{MM}
			M_0= \max\{|a_0|,\ldots,|a_{n-2}|\},\quad n\geq 2.
			\end{equation}
For $ z\in G_j(R) $, where $ R $ satisfies \eqref{RR}, we have $\ell(z)\equiv 0$ if $n=1$ and
	\begin{equation}\label{ell}
	|\ell(z)| \leq \frac{1}{|z|^{2}} \sum_{s=2}^{n}\left|a_{n-s}\right| 
	\leq \frac{(n-1) M_{0}}{|z|^{2}}<1, \quad |z|>R,\quad n\geq 2.
	\end{equation}
It follows that $ 1+\ell(z) $ lies entirely in the right half-plane for all $ |z|>R $. 

In \eqref{LT} and throughout the rest of this paper, we use the following branches when  $ z $ is confined to $ G_j(R) $. For $ (1+\ell(z))^{1/2} $ and $ (1+\ell(z))^{1/4} $, we always use the principal branch 
	\begin{equation}\label{PB}
	 -\pi< \arg (1+\ell(z)) \le \pi.
	\end{equation}
When an expression $ w =w(z)$ does not represent  $ 1+ \ell(z) $ and $ k $ is a positive integer,  $w^{k/2}$ will be defined by $w^{k/2}=(w^{1/2})^k$, where 
\begin{equation}\label{SR}
w^{{1}/{2}}=|w|^{{1}/{2}} \exp\left( i\frac{\arg (w)}{2}\right),
\quad \psi_j - \pi <\arg (w)\leq \psi_j+\pi,
\end{equation}
and $w^{k/4}$ will be defined by $w^{k/4}=(w^{1/4})^k$, where
\begin{equation}\label{FR}		
w^{{1}/{4}}=|w|^{{1}/{4}}\exp \left( i\frac{\arg (w)}{4}\right),
\quad \psi_j - \pi <\arg (w)\leq \psi_j+\pi.
\end{equation}
For more information regarding these branches for $ w^{1/2} $ and $ w^{1/4} $, see Appendix~\ref{appendix}.

It follows from \eqref{ell} that
	$$
	 (1+\ell(z))^{1/2} \quad \text{and} \quad (1+\ell(z))^{1/4} 
	 $$
are analytic in $ \{z:|z|>R\} $.
Then the functions 
	\begin{equation}\label{AB}
	A(z)=z^{{n}/{2}}\left( 1+\ell \left( z\right) \right) ^{{1}/{2}}
	\quad\textnormal{and}\quad 
	B(z)=z^{{n}/{4}}\left( 1+\ell \left( z\right) \right)^{{1}/{4}}
	\end{equation}
are analytic in the domain $\left\{ z:|z|>R,\, \psi_j-\pi <\arg (z)<\psi_j +\pi \right\}$. In particular, they are analytic in $ G_j(R) $, and with the choice of $ z_0 $ and the path of integration as above, it follows that $ L_j(z) $ is analytic in $ G_j (R)$ as well. Liouville's transformation \eqref{LT} is related to the polynomial $Q(z)$ in \eqref{poly_Q} by means of
	$$
	A(z)^2=Q(z)=B(z)^4.
	$$

In Section \ref{error-term-sec} we will find an asymptotic representation for $L_j(z)$ in $G_j(R)$,
which will be used to prove that $L_j(z)$ is one-to-one in Section~\ref{uniqueness-sec}. The latter
shows that $w_j(\zeta)$ in \eqref{LT} is well-defined. In Section~\ref{perturbed-sine} we then
proceed to show that Liouville's transformation \eqref{LT} transforms equation \eqref{s-e1} into a perturbed sine equation of the form
	\begin{equation}\label{e2}
	w''_j(\zeta )+\left[ 1-T(\zeta )\right] w_j(\zeta )=0,		
	\end{equation}
where $ T(\zeta)=O(\zeta^{-2}) $ as $ \zeta \to\infty$.  The interplay between \eqref{s-e1} and \eqref{e2} turns out to be crucial in Section~\ref{asymp-int}, where we will show that the solutions of \eqref{e2}
are asymptotic to the three known types of solutions $e^{iz}, e^{-iz}, \sin(z-z_0)$ of the sine equation \eqref{sine-equation}. Finally, in Section~\ref{geometric-sec} we
will discuss a geometric property satisfied by $L_j(z)$, which will be needed in proving Theorem~\ref{theo-zeros}.

\subsection{\sc Asymptotic growth of $ L_j(z) $}\label{error-term-sec}

Lemma~\ref{L2} below gives an asymptotic representation for the function $L_j(z)$ in \eqref{LT}
when $z$ is confined to $G_j(R)$. This representation will be used later to prove that the function $w_j(\zeta)$ in \eqref{LT} is well-defined, among other things.

\begin{lemma}\label{L2}
	The function $\zeta =L_j(z)$ in \eqref{LT} satisfies
	\begin{equation}\label{asL}
	\zeta =L_j(z)=\frac{2}{n+2} z^{(n+2)/2}\left( 1+K(z)\right) ,\quad z\in G_j(R),  
	\end{equation}
	where  $|K(z)|=O(e_n(|z|))$ as $z\rightarrow \infty $ in $G_j(R)$, where $e_n(r)$ is in \eqref{er}. Moreover, we have $|\z|\sim \frac{2}{n+2}|z|^{(n+2)/2}$, as $|z|\to\infty$, and 
	\begin{equation}\label{K(z)}
	|\arg(1+K(z))| =O(e_n(|z|)), \quad |z| \to \infty.
	\end{equation}
\end{lemma}

\begin{proof}
	From \eqref{ell}, \eqref{PB} and the binomial series \cite[pp.~118--119]{re}, the function  $(1+\ell(z))^{1/2}$  can be expressed, when $ n\ge 2$,  as 
	\begin{equation}\label{e5}
	(1+\ell(z))^{1/2} =\sum_{k=0}^{\infty }\alpha _{k} \ell(z)^k= \sum_{k=0}^{\infty }\alpha _{k}\left( \frac{a_{n-2}}{z^{2}}+\frac{a_{n-3}}{z^{3}}+\cdots +\frac{a_{0}}{z^{n}}\right)^{k},\quad |z|>R,
	\end{equation}
	where $\alpha_0=1$ and
	$$
	\alpha _{k}=\frac{1}{k!}\left( \frac{1}{2}\right) \left( \frac{1}{2}
	-1\right) \cdots \left( \frac{1}{2}-k+1\right) ,\quad k\geq 1.
	$$
	After rearranging the terms in the last sum in \eqref{e5}, we obtain 
	\begin{equation}\label{e6}
	(1+\ell(z))^{1/2}= 1+\frac{c_{2}}{z^{2}}+\frac{c_{3}}{z^{3}}+\cdots +\frac{c_{m+1}}{z^{m+1}}+\cdots,  \quad |z|>R,
	\end{equation}
	where $c_{2}=\alpha_{1}a_{n-2}$, $c_{3}=\alpha _{1}a_{n-3}$, $c_4= \alpha_1 a_{n-4} + \alpha_2 a_{n-2}^2$, and so on. 
	
	From \eqref{LT} and \eqref{e6}, we obtain two cases:
	
	(1) If $n$ is even, say $n=2p$, where $p$ is a positive integer, then 
	\begin{align*}
	\zeta & =\int_{z_0}^{z} \left( \xi^{p}+c_{2}\xi^{p-2}+c_{3}\xi^{p-3}+\cdots 
	+c_{p+1}\xi^{-1}+\cdots \right) d\xi\\
	& =\left( \frac{1}{p+1}z^{p+1}+\frac{1}{p-1}c_{2}z^{p-1}+\cdots 
	+c_{p+1}\log z+\cdots \right)+C(z_{0})\\
	&=\frac{1}{p+1}\,z^{p+1}\left( 1+\left[\frac{p+1}{p-1}c_{2}\,\frac{1}{z^{2}}+\cdots 
	+\bigg( (p+1)c_{p+1}\log z+(p+1)C(z_{0})\bigg) 
	\frac{1}{z^{p+1}}+\cdots \right] \right),
	\end{align*}
where $C(z_{0})$ is a constant depending on $z_{0}$ and where the branch for $\log z=\log |z|+i\arg (z)$ is $ \psi_j-\pi<\arg (z) \le \psi_j+\pi$. The above terms containing $p-1$ do not appear when $p=1$, i.e., when $n=2$.

	(2) If $n\geq 3$ is odd, then the logarithmic term does not appear in the expression of $ \zeta $. In fact, we have 
	\begin{eqnarray*}
		\zeta &=&\left( \frac{2}{n+2}z^{\frac{n}{2}+1}+\frac{2}{n-2}c_{2}z^{\frac{n}{2}-1}+\cdots +\frac{2}{n-2m}c_{m+1}z^{\frac{
				n}{2}-m}+\cdots \right) +C(z_{0})\\
		&=&\frac{2}{n+2}\,z^{\frac{n}{2}+1}\left( 1+\left[ \frac{n+2}{n-2}c_{2}\,\frac{1}{z^{2}}+\cdots +\frac{n+2}{n-2m}c_{m+1}\frac{1}{z^{m+1}}+\cdots \right] +\frac{(n+2)C(z_{0})}{2}\frac{1}{z^{\frac{n}{2}+1}}\right).
	\end{eqnarray*}

From Cases (1) and (2), and the easy case when $ n=1 $, we obtain
	\begin{equation*}
	\zeta =\frac{2}{n+2}z^{(n+2)/2}\left( 1+K(z)\right), 
	\end{equation*}
where $K(z)$ is analytic in $\left\{ z:|z|>R,\, \psi_j-\pi <\arg (z)<\psi_j +\pi \right\}$. By considering the cases $n=1$, $n=2$ and $n>2$ separately, we find that $ K(z) $ satisfies $|K(z)|=O(e_n(|z|))$ as $z\rightarrow \infty $, where $ e_n(r) $ is given in \eqref{er}. This proves the representation in \eqref{asL}.
	
To see that $|\z|\sim \frac{2}{n+2}|z|^{(n+2)/2}$ as $|z|\to\infty$ in $G_j(R)$, it suffices to observe that $|K(z)|=O(e_n(|z|))$ as $z\rightarrow \infty $, where $e_n(|z|)\rightarrow 0$ as $|z|\to\infty$ by \eqref{er}.

	It remains to prove \eqref{K(z)}. Assume that $ R>0 $ is large enough such that  $|K(z)|<1$ 	whenever $|z|>R$.  If $K(z)\neq 0$, define the circle $\mathcal{C}_{z}=\{w :|w -1|=|K(z)|\}$, and let $2\varphi_{z}$ denote the angle which the circle $\mathcal{C}_{z}$ subtends at the origin.
	
	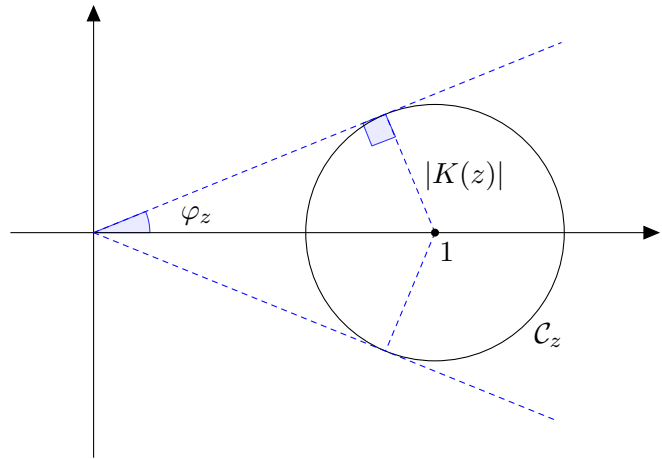
\begin{figure}[h!]
		\definecolor{ttttff}{rgb}{0.2,0.2,1}
		\definecolor{qqqqff}{rgb}{0,0,1}
		\begin{tikzpicture}[line cap=round,line join=round,>=triangle 45,x=1.0cm,y=1.0cm]
		\clip(-1.2,-3) rectangle (7.61,3.03);
		\draw [shift={(0,0)},line width=0.4pt,color=ttttff,fill=ttttff,fill opacity=0.1] (0,0) -- (0:0.74) arc (0:22.26:0.74) -- cycle;
		\draw[line width=0.4pt,color=ttttff,fill=ttttff,fill opacity=0.1] (3.54,1.45) -- (3.66,1.15) -- (3.97,1.27) -- (3.84,1.57) -- cycle; 
		\draw [->] (0,-2.98) -- (0,3.02);
		\draw [->] (-1.09,0) -- (7.45,0);
		\draw(4.49,0) circle (1.7cm);
		\draw [dash pattern=on 2pt off 2pt,color=qqqqff] (0,0)-- (6.15,2.52);
		\draw [dash pattern=on 2pt off 2pt,color=qqqqff] (0,0)-- (6.06,-2.48);
		\draw [dash pattern=on 2pt off 2pt,color=qqqqff] (4.49,0)-- (3.84,1.57);
		\draw [dash pattern=on 2pt off 2pt,color=qqqqff] (4.49,0)-- (3.84,-1.57);
		\draw (1,0.5) node[anchor=north west] {\small $\varphi_z$};
		\draw (4.2,1.1) node[anchor=north west] {\small $|K(z)|$};
		\draw (4.4,0.02) node[anchor=north west] {\small $1$};
		\draw (5.65,-1.05 ) node[anchor=north west] {\small $\mathcal{C}_{z}$};
		
		\begin{scriptsize}
		\fill [color=black] (4.49,0) circle (1.5pt);
		\end{scriptsize}
		\end{tikzpicture}
		\caption{The circle $\mathcal{C}_{z}$ and the angle $2\varphi_z$.}
		\label{fig01}
	\end{figure}
	
	We find that every $w \in \mathcal{C}_{z}$ satisfies $|\arg (w) |\leq\varphi_{z}$. In particular, $|\arg(1+K(z))|\leq\varphi_{z}$.  Therefore,
	\begin{equation*}
	|\arg(1+K(z)) |\le\varphi_{z}=\arcsin \left( \left\vert K(z)\right\vert \right) 
	\le \frac{\pi}{2}|K(z)| = O(e_n(|z|)), \quad |z| >R.
	\end{equation*}
	The lemma is now proved.
\end{proof}


\subsection{\sc Univalence of $L_j(z)$}\label{uniqueness-sec}

In this section we prove that the function $L_j(z)$ in \eqref{LT} is one-to-one (univalent) in $G_j(R)$. 
This proves that $w_j(\zeta)$ in \eqref{LT} is well-defined, so that the transformation of \eqref{s-e1} into 
\eqref{e2} is justified.

For any $ j\in \{0, \ldots, n+1\} $, the function
	\begin{equation}\label{w}
	u(z)=\frac{2}{n+2} z^{(n+2)/2}
	\end{equation}
has the following two properties:
\begin{enumerate}
	\item[(i)] $ u(z) $ maps the critical ray $ \arg (z) = \psi_j $ in \eqref{s-critical rays} onto the positive real axis 
	if $ j $ is even and onto the negative real axis if $ j $ is odd.
	\item[(ii)]  $ u(z) $ is univalent in each sector
		\begin{equation}\label{U}
		U_j=\left\{z\in\C : |z|>0,\ \psi_{j-1}<\arg (z)<\psi_{j+1}\right\},
		\end{equation}
	and maps each $U_j$ onto the slit plane $ V^+ $ if  $j$ is even or onto $ V^- $ if $j$ is odd, where
		\begin{equation}
		V^+ = \C \setminus \left\{w\in\R: w \le 0 \right\}\quad \text{and} \quad
		V^- =\C \setminus \left\{w\in\R: w \ge 0 \right\}.
		\end{equation}
\end{enumerate}

We proceed to prove that $L_j(z)$ in \eqref{LT} has similar properties in $G_j(R)$ as $u(z)$ 
has in $U_j$, that is, $L_j(z)$ is univalent in $ G_j(R) $ and maps $ G_j(R) $ onto a domain containing
	\begin{equation}\label{image G_j}
		\tilde{G}_j(\delta,\tilde{R}) 
		= \left\{\zeta\in\C : |\zeta|>\tilde{R},\ |\arg(\zeta)-\pi j|\leq\pi-\delta\right\},
	\end{equation}
for a given small $ \delta>0 $. This implies that the function $w_j(\zeta)$ in \eqref{LT} is analytic in a domain containing $ \tilde{G}_j(\delta,\tilde{R}) $. The following lemma is similar to \cite[Lemma~4.3.6]{langley0}.

\begin{lemma}\label{L3-new}
For any $ j\in \{0,1, \ldots, n+1\} $, the following two properties hold.
\begin{itemize}
\item[\textnormal{(1)}]  The function $L_j(z)$ in \eqref{LT} is one-to-one in the domain $ G_j(R) $, provided that $ R>0 $ is sufficiently large.
	
\item[\textnormal{(2)}] Let $ \delta>0 $ be small enough. Then there exists $ \tilde{R}>0 $ large enough such that $ L_j(G_j(R)) $ contains $ \tilde{G}_j(\delta,\tilde{R})  $ defined in \eqref{image G_j}.
	\end{itemize}
\end{lemma}

\begin{proof}
	(1) From \eqref{LT} and \eqref{w}, we have
		$$
		\frac{dL_j}{dz}=z^{n/2} (1+\ell(z))^{1/2} = \frac{du}{dz}(1+\ell(z))^{1/2}, \quad z\in G_j(R).
		$$
	Combining this with \eqref{e6}, we deduce that
		$$
		\frac{dL_j}{du}=\frac{dL_j}{dz}\frac{dz}{du}=(1+\ell(z))^{1/2}=1+H(z),\quad z\in G_j(R),
		$$
	where $ H(z) = O(z^{-2}) $. We may assume that $R\geq R_0$ is large enough so that $|H(z)|<1/2$ for $z\in G_j(R)$. Suppose that $z_1,z_2\in G_j(R)$ are distinct, and set $u_k=u(z_k)$. Integrating along any contour in $ G_j(R) $ joining $ z_1 $ and $ z_2 $, it follows by the univalency of $u$ that
		\begin{eqnarray*}
			L_j(z_1)-L_j(z_2)&=&\int_{z_2}^{z_1}\frac{dL_j}{dz}\, dz=\int_{u_2}^{u_1}\frac{dL_j}{du}\, du
			=\int_{u_2}^{u_1}(1+H(z))\, du,
		\end{eqnarray*}
	from which
		$$
		|L_j(z_1)-L_j(z_2)|\geq \left|\int_{u_2}^{u_1}\, du\right|-\int_{u_2}^{u_1}|H(z)|\, |du|\\
		\geq |u_1-u_2|/2>0.
		$$	
	This shows that $L_j(z)$ is univalent in $G_j(R)$.
	
	(2) Take $ R> 0 $ large enough such that $K(z)$ in \eqref{asL} satisfies $ |K(z)|<1 $ for $z\in G_j(R)$. By the discussion related to  \eqref{AB} 
	the map $ \zeta = L_j(z) $ is analytic in $ G_j(R) \cup \partial G_j(R) $. The function $ u(z) $ 
	in \eqref{w} maps the boundary $ \partial G_j(R) $ onto the boundary of the domain
		\begin{equation}\label{Du}
		\left\{u: |u|> \frac{2}{n+2} R^{(n+2)/2},\,  |\arg (u) - \pi j |< \pi \right\}.
		\end{equation} 
	By Lemma~\ref{L2} we may write $ L_j(z) = u(z) (1+K(z)) $. Thus the image of $ G_j(R) $ under $ L_j(z) $ is the domain \eqref{Du} with a deformation on the boundary by the factor of $1+K(z) $. First, when $ z $ wanders along the circular part of $ \partial G_j(R) $, we see that $ L_j(z) $ traces a simple curve, which lies entirely in the disc $D(0, \tilde{R}_1)$ of radius $ \tilde{R}_1> \frac{4}{(n+2)}R^{(n+2)/2} $. Second, when $ z $ wanders along the critical rays $ \arg (z) = \psi_{j\pm 1} $, we see from \eqref{asL} that
		\begin{equation*}
		\begin{split}
		\arg(\zeta) &= \frac{n+2}{2} \arg (z) + \arg(1+K(z))\\
		&= \pi (j\pm 1) +\arg(1+K(z)).
		\end{split}
		\end{equation*}
Hence, we obtain from \eqref{K(z)} and \eqref{er} that
		$$
		|\arg(\zeta) - \pi (j\pm 1)| = |\arg(1+K(z))|\to 0  
		$$
as $|z|\to \infty$ such that $\arg (z) = \psi_{j\pm 1} $.	Then there exists $ \tilde{R}> \tilde{R}_1 $ such that 
		$$
		|\arg(\zeta) - \pi (j\pm 1)|< \delta, \quad |\zeta|>\tilde{R}.
		$$
From this, one can see that the boundary of the image $ L_j(G_j(R)) $  of $ G_j(R) $  is outside the region  $ \tilde{G}_j(\delta, \tilde{R}) $  in \eqref{image G_j} and therefore, we easily deduce the assertion.
\end{proof}


\subsection{\sc The perturbed sine equation}\label{perturbed-sine}

Lemma~\ref{Q-P} below shows that \eqref{LT} transforms equation \eqref{s-e1} into equation \eqref{e2}.
The result is briefly stated in \cite[p.~180]{hille2}. Our proof relies on simple manipulations
on the functions $A(z),B(z)$ in \eqref{AB}.

\begin{lemma}\label{Q-P}
Let $g(z)\not\equiv 0$ be a solution of \eqref{s-e1}, and let $ g_j(z) $ be its restriction to $ G_j(R) $. Then $w_j(\zeta )$ defined in \eqref{LT} satisfies an equation of the form \eqref{e2}, where $ T(\zeta) $  is analytic in a domain containing the region $ \tilde{G}_j(\delta, \tilde{R}) $, and 
		\begin{equation}\label{T}
		T(\zeta ) = \frac{1}{4}\left( \frac{Q^{\prime \prime }(z)}{Q(z)^{2}}
		-\frac{5}{4}\frac{Q^{\prime}(z)^{2}}{Q(z)^{3}}\right) = O\left( \frac{1}{\zeta^{2}}\right) ,\quad \zeta \rightarrow
		\infty .
		\end{equation}
\end{lemma}

\begin{proof}
	From the first formula in \eqref{LT}, we have
	\[
	\zeta ^{\prime }(z)=\frac{dL_j(z)}{dz}=z^{{n}/{2}}\left( 1+\ell
	\left( z\right) \right) ^{{1}/{2}}=A(z),\quad z\in G_j(R).
	\]
Since $B(z)=z^{{n}/{4}}\left( 1+\ell \left( z\right) \right) ^{{1}/{4}}$ is analytic in $G_j(R)$ and  
$B(z)^{4}=z^{n}\left( 1+\ell \left(z\right) \right) =Q(z)$, we~get
		\begin{equation}\label{BB}
		B^{\prime }(z)=\frac{Q^{\prime }(z)}{4B(z)^{3}},\quad z\in G_j(R).  
		\end{equation}
Therefore, from the second formula in \eqref{LT}, we obtain
	\begin{eqnarray*}
	w^{\prime }_j(\zeta ) &=& \frac{dw_j(\zeta)}{d\zeta}
		=\frac{(d/dz)\left( B(z)g_j(z)\right) }{d\zeta/dz}
		=\frac{g_j(z)B^{\prime }(z)+B(z)g^{\prime }_j(z)}{\zeta ^{\prime }(z)} \\
		&=&\frac{Q^{\prime }(z)}{4B(z)^{3}A(z)}g_j(z)+\frac{B(z)}{A(z)}g^{\prime }_j(z),\\
	w^{\prime \prime }_j(\zeta ) 
		&=&\frac{dw'_j(\zeta)}{d\zeta}=\frac{1}{d\zeta/dz}\frac{d w'_j(\zeta)}{dz}  \\
		&=&\frac{1}{A(z)}\bigg\{ \frac{1}{4}g_j(z)\frac{d}{dz}\left(\frac{Q^{\prime
			}(z)}{B(z)^{3}A(z)}\right)\\ 
		&& +\left[ \frac{Q^{\prime }(z)}{4B(z)^{3}A(z)}+%
		\frac{d}{dz}\left( \frac{B(z)}{A(z)}\right) \right] g^{\prime }_j(z) +\frac{B(z)}{A(z)}g^{\prime \prime }_j(z)\bigg\}, 
	\end{eqnarray*}%
where
	\[
	\frac{d}{dz}\left( \frac{B(z)}{A(z)}\right) =\frac{B^{\prime }(z)}{A(z)}-%
	\frac{B(z)A^{\prime }(z)}{A(z)^{2}}.
	\]
From $A(z)^{2}=Q(z)$, we obtain $2A(z)A'(z)=Q'(z)$. It then follows from \eqref{BB}  that
	\begin{eqnarray*}
		\frac{d}{dz}\left( \frac{B(z)}{A(z)}\right)  &=&\frac{Q^{\prime }(z)}{%
			4B(z)^{3}A(z)}-\frac{B(z)Q^{\prime }(z)}{2A(z)^{3}} \\
		&=&\left( \frac{A(z)^{2}}{4B(z)^{3}A(z)^{3}}-\frac{2B(z)^{4}}{%
			4B(z)^{3}A(z)^{3}}\right) Q^{\prime }(z) \\
		&=&\left( \frac{Q(z)}{4B(z)^{3}A(z)^{3}}-\frac{2Q(z)}{4B(z)^{3}A(z)^{3}}
		\right) Q^{\prime }(z) \\
		&=&-\frac{Q(z)}{4B(z)^{3}A(z)^{3}}Q^{\prime }(z)=-\frac{Q^{\prime }(z)}{%
			4B(z)^{3}A(z)}.
	\end{eqnarray*}
Thus
	\begin{equation}\label{proof1}
	w^{\prime \prime }_j(\zeta )=\frac{1}{4}\frac{1}{A(z)}g_j(z)\frac{d}{dz}\left( 
	\frac{Q^{\prime }(z)}{B(z)^{3}A(z)}\right) +\frac{B(z)}{Q(z)}g^{\prime
		\prime }_j(z).
	\end{equation}
	Substituting $g^{\prime \prime }_j(z)=-Q(z)g_j(z)$ and using the second formula in \eqref{LT} yields
	\begin{eqnarray*}
		w^{\prime \prime }_j(\zeta ) &=&\frac{1}{4}\frac{1}{A(z)}g_j(z)\frac{d}{dz}\left(
		\frac{Q^{\prime }(z)}{B(z)^{3}A(z)}\right) -B(z)g_j(z) \label{proof2}\\
		&=&\left( \frac{1}{4B(z)A(z)}\frac{d}{dz}\left( \frac{Q^{\prime }(z)}{%
			B(z)^{3}A(z)}\right) -1\right) B(z)g_j(z) \nonumber \\
		&=&\left( T(\zeta) -1\right) w_j(\zeta ), \nonumber
	\end{eqnarray*}%
where 
	\begin{eqnarray}
	T(\zeta )&=&\frac{1}{4B(z)A(z)}\frac{d}{dz}\left( \frac{Q^{\prime }(z)}{B(z)^{3}A(z)}\right)\label{proof3}\\
	&=&\frac{1}{4B(z)A(z)}\left(\frac{Q^{\prime \prime }(z)}{B(z)^{3}A(z)}
	-\frac{Q^{\prime }(z)\left[ B(z)^{3}A(z)\right] ^{\prime }}{B(z)^{6}A(z)^{2}}\right) \nonumber \\
		&=&\frac{Q^{\prime \prime }(z)}{4Q(z)^{2}}-\frac{Q^{\prime }(z)\left[
			B(z)^{3}A(z)\right] ^{\prime }}{4Q(z)^{2}A(z)B(z)^3}.\nonumber
	\end{eqnarray}%
Differentiating $\left[ B(z)^{3}A(z)\right] ^{4}=Q(z)^{5}$, we have
	\[
	\left[ B(z)^{3}A(z)\right] ^{\prime }=\frac{5}{4}\frac{Q^{\prime }(z)Q(z)^{4}%
	}{\left( B(z)^{3}A(z)\right) ^{3}}=\frac{5}{4}\frac{Q^{\prime }(z)Q(z)}{%
		B(z)A(z)}, 
	\]
and so
	\begin{eqnarray*}
		T(\zeta ) &=&\frac{1}{4}\left( \frac{Q^{\prime \prime }(z)}{Q(z)^{2}}
		-\frac{5}{4}\frac{Q^{\prime }(z)^{2}}{Q(z)^{3}}\right).
	\end{eqnarray*}%
	Hence, from Lemma~\ref{L2} and \eqref{L3-new}, it follows that  $ T(\zeta) $  analytic in a domain containing the region $ \tilde{G}_j(\delta, \tilde{R}) $ and 
	\[
	T(\zeta )=O\left( \frac{1}{z^{n+2}}\right) =O\left( \frac{1}{\zeta^{2}
	}\right) ,\text{\quad }\zeta \rightarrow \infty .
	\]%
	This completes the proof.
\end{proof}

\subsection{ \sc A geometric property of $ L_j(z) $}\label{geometric-sec}

The following lemma about the pre-images of horizontal lines under $L_j(z)$ will be needed for proving 
Theorem~\ref{theo-zeros} in Section~\ref{proof-of-thm2}. 

\begin{lemma}\label{L1}
Let $\delta,R,\tilde{R}$ be as in Lemma~\ref{L3-new}, and let $v_0\in\R$.  For  $ j \in \{0,\ldots, n+1\}$, let $\ell_j: \zeta =(-1)^ju+iv_0$ 	denote a horizontal half-line in $ \tilde{G}_j(\delta, \tilde{R}) $, where 
		$$
		\left\{
		\begin{array}{lcl}
		u\ge 0,\ & \text{if} & |v_0|>\tilde{R},\\
		u>\sqrt{\tilde{R}^2-v_0^2},\ & \text{if} & |v_0|\le \tilde{R}.
		\end{array}
		\right.
		$$
	Then there exists a constant $ C=C(n,|v_0|,R)>0 $ such that the pre-image $\el_j$ of $\ell_j$ under $\z=L_j(z)$ is a curve lying in the domain	
		\begin{equation*}
		\Lambda^*_{j}=\{z=re^{i\theta} : |\theta-\psi_j|<Ce_n(r),\,r>R\}.
		\end{equation*}
	
\end{lemma}

\begin{proof}
	Let $\z=\rho e^{i\phi}\in \ell_j$, and let $z=re^{i\theta}$ be the unique pre-image of $\z$
	in $G_j(R)$, provided by Lemma~\ref{L3-new}. 
	From the geometry, we may now write more precisely that 
	$$
	|\phi-\pi j| = \arcsin\left( \frac{|v_0|}{\rho}\right).
	$$ 
We assume that $ R $ is large enough so that $ |K(z)|\le 1/2 $ for all $ |z|>R $. Then we obtain from Lemma~\ref{L2} that $|\z|\ge (n+2)^{-1} |z|^{(n+2)/2}$ or $\rho\ge (n+2)^{-1} r^{(n+2)/2}$ for all $ r>R $. This yields
	\begin{equation*}
	|\phi-\pi j|\leq \frac{\pi|v_0|}{2\rho}\leq \frac{(n+2)\pi |v_0|}{2}r^{-(n+2)/2}, 
	\quad r>R.
	\end{equation*}
%
By recalling \eqref{er}, we see that $ r^{-(n+2)/2} \le e_n(r) $ holds for all $ r> R $ and for every $ n\geq 1$. Hence,
	\begin{equation}\label{pr-01}
	|\phi-\pi j|\le\frac{(n+2)\pi |v_0|}{2} e_n(r), \quad  r> R.
	\end{equation}
We deduce from \eqref{K(z)} that there exists a constant $ C_0(R)>0 $ such that
		\begin{equation}\label{KK}
		|\arg(1+K(z))| \le C_0(R) \, e_n(r), \quad  r> R.
		\end{equation}
From \eqref{asL}, it follows that
	\begin{equation*}\label{phi}
	\phi=\frac{n+2}{2}\theta+\arg(1+K(z)).
	\end{equation*}
Then using $\psi_j =\frac{2\pi j}{n+2}$ yields
	\begin{eqnarray*}
		|\theta - \psi_j|\leq \frac{2}{n+2}\Big(|\phi-\pi j|+|\arg(1+K(z))|\Big).
	\end{eqnarray*}
Combining this with \eqref{pr-01} and \eqref{KK}, we get
	$$
	|\theta - \psi_j|\le C e_n(r),\quad r>R,
	$$
where $ C=\pi|v_{0}|+{2\pi}{(n+2)^{-1}} \, C_{0}(R) $.  This completes the proof of the lemma.
\end{proof}
%



\section{Asymptotic integration theory}\label{asymp-int}


The discussion on asymptotic integration related to 
equation \eqref{e2}  in Hille's book \cite[\S7.4]{hille} is conducted on Riemann surfaces. Here we confine the presentation and justification to the complex plane, and at the same time more details will be given. This will make this theory more accessible to a wider mathematical community.

To avoid ambiguity, we use $z$ as an independent variable instead of $\zeta$, since this 
section is of independent interest. However, in proving the main theorems, we will apply the results 
of this section to equation \eqref{e2} via Liouville's transformation \eqref{LT}, and at that point 
the variable $\z$ will be used instead of $z$.

\subsection{\sc Preparations}\label{further-sec}
 
We begin with Gronwall's lemma on unbounded intervals, which plays a key role in asymptotic integration. The following version is stated as a problem in \cite[p.~20]{hille}. We give a complete proof for the convenience of the reader.

\begin{lemma}\label{theo1}
	For $a\in\R$, let $K(t)>0$ be an integrable function on $(a,\infty)$, and let $g(t)\geq 0$ be a continuous and bounded function on $[a,\infty)$. If $f(t)\geq 0$ is a continuous and bounded function on $[a,\infty)$, and if
	\begin{equation}\label{1.1}
	f(t)\le g(t)+ \int_t^\infty K(s) f(s)\, ds,\quad t\geq a,
	\end{equation}
	then
	\begin{equation}\label{1.2}
	f(t)\le g(t)+\int_t^\infty K(s) \exp\left\{\int_t^s K(u) du\right\} g(s)\, ds.
	\end{equation}
\end{lemma}

\begin{proof}
	Set
	\begin{equation*}
	F(s)=-\int_s^\infty K(u) f(u) du,\quad s\in (a,\infty).
	\end{equation*}
	By differentiating and then using \eqref{1.1}, we obtain
	\begin{equation*}
	F'(s)=K(s)f(s)\le K(s)g(s)-K(s)F(s),
	\end{equation*}
	that is, $ 
		F'(s)+K(s)F(s) \le K(s) g(s). 
		$	
	Multiplying both sides by
		\[
		\exp\left\{-\int_s^\infty K(u) du\right\}>0,
		\]
	we get
		\begin{equation*}
		\frac{d}{ds} \left(F(s)\exp\left\{-\int_s^\infty K(u)\, du\right\} \right) \le K(s)\exp\left\{-\int_s^			\infty K(u)\, du\right\} g(s).
		\end{equation*}
	Next, since $\displaystyle\lim_{s\to\infty}F(s)=0$, an integration from $t\geq a$ to $\infty$, and	using the fundamental theorem of calculus, results in
	\begin{equation*}
	-F(t)\exp\left\{-\int_t^\infty K(u)\, du\right\}
	\le \int_t^\infty K(s)\exp\left\{-\int_s^\infty K(u)\, du\right\} g(s)\, ds,
	\end{equation*}
	or, in other words,
	\begin{equation}\label{1.8}
	-F(t)\le \int_t^\infty K(s)\exp\left\{\int_t^s K(u)\, du\right\} g(s)\, ds.
	\end{equation}
	From \eqref{1.1}, we have $f(t)\le g(t) - F(t)$, thus \eqref{1.2} follows from \eqref{1.8}.
\end{proof}

We now discuss frameworks that are needed in the theory of asymptotic integration. 
Corresponding to \eqref{e2}, we consider a more general perturbed sine  equation of the form
	\begin{equation}\label{1}
	w''+(1-F(z))w=0,
	\end{equation}
where $F(z)$ satisfies either Hypothesis $ \mathbf{F}^+ $ or Hypothesis $ \mathbf{F}^- $ below.
\medskip
\begin{quote}
	\textbf{Hypothesis $\mathbf{F}^+$.}
	\emph{ The function $F(z)$ is analytic in a domain 
		$$
		D^{+}(\delta_0,R_0)=\{z \in \mathbb{C}:|z|>R_0, |\arg(z)|< \pi- \delta_0\},
		$$ 
	where $ \delta_0 \in (0, \pi) $ and $ R_0>0 $.
	For each $ z\in D^+(\delta_0, R)$, where $R\geq R_0/\sin (\delta_0)$, the integral
	$		\int_z^\infty |F(t)| |dt| $ 
		exists along the path of  integration given by $ t=z+r $, $ 0\le r<\infty $. 
		Moreover, there exists a $ \delta $ satisfying $\delta>\delta_0$ such that 
		\begin{equation}\label{trimming-edges}
		\lim_{R\to\infty} \sup_{z\in D^+(\delta,R)} \int_z^\infty |F(t)| |dt|=0.
		\end{equation}}

		

\end{quote}

\begin{figure}[h!]
\centering 

\definecolor{ffqqqq}{rgb}{1,0,0}
\definecolor{qqqqff}{rgb}{0,0,1}
\begin{tikzpicture}[line cap=round,line join=round,>=triangle 45,x=1.0cm,y=1.0cm]
\clip(-5.58,-4.34) rectangle (5.64,4.28);

\begin{scope}
\begin{pgfinterruptboundingbox} 
\path[invclip](0,0) -- (-8,3) -- (-8,-3) -- cycle;
\end{pgfinterruptboundingbox}
\draw [line width=0.4pt,fill=black,fill opacity=0.1, draw opacity =0] (0,0) circle (1.18cm);
\draw[dashed](0,0) circle (3.6cm);
\end{scope}

\draw [shift={(0,0)}] (159.41:1.3) arc (159.41:180:1.3);

\draw (1.05,0) node[anchor=north west] {$R_0$};
\draw (3.6,0) node[anchor=north west] {$R$};
\draw (-1.9,0.6) node[anchor=north west] {$\delta_0$};
\draw (-4.3,1) node[anchor=north west] {$\frac{R_0}{\sin (\delta_0)}$};

\draw[fill=black,fill opacity=0.1,draw opacity =0 ](0,0) -- (-8,3) -- (-8,-3) -- cycle;

\draw [
color=qqqqff,domain=-3.5:5] plot(\x,2);



\draw [line width=1.2pt,color=qqqqff] (0,0)-- (-3.14,1.18);
\draw [-latex,line width=0.3pt] (-3.27,0.63) -- (-2.03,0.76);
\draw [->,line width=0.4pt] (-5.1,0) -- (5.04,0);
\draw [->,line width=0.4pt] (0,-4) -- (0,4);

\begin{scriptsize}
\fill [color=qqqqff] (-3.14,1.18) circle (1.5pt);
\fill [color=qqqqff] (-3.5,2) circle (1.5pt);
\fill [color=black] (0,1.18)  circle (1.5pt);
\end{scriptsize}
\draw (-3.9,2.5) node[anchor=north west] {$z$};

\draw [dash pattern=on 2pt off 2pt, color =blue] (-3.14,1.18) -- (0,1.18);

\end{tikzpicture}
\caption{Geometric justification for the inequality $ R\ge R_0 / \sin(\delta_0) $.}\label{path}
\end{figure}
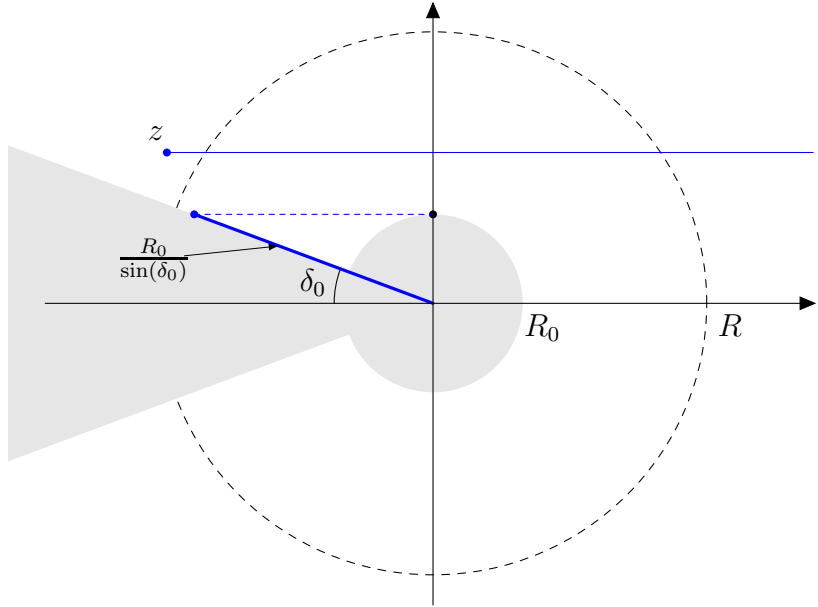

\medskip
Calling the above conditions Hypothesis~$\mathrm{F}^+$ is inspired by the notation and conditions in \cite[Ch.~7.4]{hille}.
The assumption \eqref{trimming-edges} contains the idea of trimming the domain along the boundary, as discussed in
\cite[p.~335]{hille}. 
%
%
In particular, if $ F(z) $ is analytic and if $F(z)=O(z^{-2})$ as $z\to\infty$ in $ D^+(\delta,R) $, then Hypothesis $\mathrm{F}^+$ is satisfied. This is the situation with the coefficient function in \eqref{e2} by Lemmas~\ref{L3-new} and \ref{Q-P}, of which the former guarantees that the
variable $\zeta$ in \eqref{e2} belongs to a domain that contains $ D^+(\delta,R) $. For this example, which is the foundation of this paper, $ \delta$ and  $\delta_0 $ can both be arbitrarily close to zero. 
If we choose $ \delta \in (\pi/2, \pi)$, then the transcendental function $F(z)= e^{-z} $ satisfies Hypothesis~$\mathrm{F}^+$.



We define Hypothesis $\mathrm{F}^-$ analogously to Hypothesis $\mathrm{F}^+$.

\begin{quote}
	\textbf{Hypothesis $\mathbf{F}^-$.}
	\emph{The function $F(z)$ is analytic in a domain 
	$$
	D^{-}(\delta_0,R_0)=\{z \in \mathbb{C}:|z|>R_0, |\arg(z)-\pi|< \pi- \delta_0\},
	$$ 
	where $ \delta_0\in (0, \pi) $ and $ R_0>0 $. For each $ z\in D^-(\delta_0, R)$, where $R\geq R_0/\sin (\delta_0)$, the integral
	$		\int_z^\infty |F(t)| |dt| $ 
	exists along the path of  integration given by $ t=z-r $, $ 0\le r<\infty $. 
	Moreover, there exists a $ \delta $ satisfying $\delta>\delta_0$ such that 
	\begin{equation*}
	\lim_{R\to\infty} \sup_{z\in D^-(\delta,R)} \int_z^\infty |F(t)| |dt|=0.
	\end{equation*}}
\end{quote}

\medskip

Regarding Hypotheses  $\mathrm{F}^+$ and  $\mathrm{F}^-$, the constants $ \alpha=\alpha(\delta,R) $ and $ \beta=\beta(\delta,R) $ defined~by
	\begin{equation}\label{M}
	\alpha = \sup_{z\in D^+(\delta,R)} \int_z^\infty |F(t)| |dt| \quad \text{and} \quad \beta = \sup_{z\in D^-(\delta,R)} \int_z^\infty |F(t)| |dt|
	\end{equation}
will be smaller than any preassigned positive number, provided that $ R>0 $ is large enough.


\subsection{\sc General result on asymptotic solutions}\label{general result}
In the theory of asymptotic integration, it turns out to be crucial that the solutions of 
\eqref{1} are also solutions of the singular Volterra integral equation \eqref{vol-equation} below. 
As the first result in the theory of asymptotic integration,
Hille proves \cite[Theorem~7.4.1]{hille}. Later on in the theory, Hille uses a ''reverse version'' of \cite[Theorem~7.4.1]{hille} without actually stating or proving it.  The following result is this reverse version that we need, which has a direct role in the theory of asymptotic integration.

\begin{theorem}\label{volterra}
	Suppose that $ F(z) $ satisfies Hypothesis $\mathrm{F}^+ $, and let $ w_{\sin}(z) $ be a non-trivial solution of the sine equation
	\begin{equation}\label{3}
	w'' + w =0.
	\end{equation}
	Then the singular Volterra integral equation
	\begin{equation}\label{vol-equation}
	w(z) = w_{\sin}(z) + \int_z^\infty \sin(t-z) F(t) w(t)dt,
	\end{equation}
	where $ z\in D^+(\delta,R) $ and the path of integration is $ t-z = r $, $ 0\le r< \infty $, has a unique solution $ w(z) $ which is a solution of  equation \eqref{1}. Moreover, with $ z=x+iy $ we have
	\begin{equation}\label{est-w}
	\left\vert w(z) -w_{\sin}(z) \right\vert \leq
	M(y) \left\{ \exp \left[ \int_{x}^{\infty }\left\vert F(s+iy) \right\vert ds\right] -1\right\} , \quad z\in D^+(\delta,R),
	\end{equation}
	where $ M(y) = \sup\limits_{s\ge x} |w_{\sin}(s+iy)| $.
\end{theorem}

\begin{proof}
	As in \cite{hille0}, we use the classical method of successive approximations. Define a sequence of functions $ \{w_n(z)\}$ by
	\begin{equation}\label{w_n}
	\left\{
	\begin{split}
	w_0(z) &= w_{\sin}(z),\\
	w_n(z) &= w_{\sin}(z)+\int_{z}^{\infty} \sin (t-z) F(t) w_{n-1}(t)\, dt, 
	\quad z\in D^+(\delta,R), \ n\ge 1, 
	\end{split}
	\right.
	\end{equation}
	where the path of integration is $ t-z = r $, $ 0\le r< \infty $. 
	
	First, we show by induction that the functions $w_n(z)$ are analytic in $D^+(\delta,R)$. It suffices to prove that  $ w_n(z) $ is  bounded in the domain
	$$
	\Sigma_\sigma = \{z \in D^+(\delta,R) : |\im (z)|< \sigma\},
	$$
	for any fixed $ \sigma>0 $. If $ n=0 $, then the function $ w_0(z)=w_{\sin}(z) $ is a solution of the sine equation \eqref{3}, and hence  there exists a constant $ C_0=C_0(\delta)>0 $ such that $ |w_{\sin}(z)|\le C_0 $ for all $ z\in \Sigma_\sigma $. Suppose that there exists a constant $ C_{n-1}=C_{n-1}(\delta)>0 $ such that $ |w_{n-1}(z)|\le C_{n-1} $ for all $ z\in \Sigma_\sigma $. Then it follows from \eqref{w_n} that
	\begin{equation*}
	\begin{split}
	|w_n(z)| &\le |w_{\sin}(z)|+\int_{z}^{\infty} |\sin (t-z)| |F(t)| |w_{n-1}(t)| |d t|,\\
	& \le  C_0+ C_{n-1}\int_{z}^{\infty} |F(t)|  |d t|, \quad z\in\Sigma_\sigma,
	\end{split}
	\end{equation*}
	since $ |\sin (t-z)| \le 1$ along the path of integration  $ t-z=r $, $ 0\le r< \infty $. From \eqref{M}, we obtain $|w_n(z)| \le C_n$ for all $ z\in \Sigma_\sigma $, where $ C_n = C_0 + \alpha C_{n-1} $. By induction, each function $ w_n(z) $ is bounded in $\Sigma_\sigma$, and therefore is analytic in $D^+(\delta,R)$. 
	
	Second, we prove that for all $ n \ge 1 $,
	\begin{equation}\label{estimate w_n}
	|w_{n}(z) - w_{n-1}(z)| \le M(y) \frac{1}{n!} \left[\int_z^\infty |F(t)| |dt|\right]^n, \quad z\in D^+(\delta,R),
	\end{equation}
	where $ M(y) = \max_{s\ge \re (z)} |w_{\sin}(s+iy)| $ and $ y=\im (z) $. From \eqref{w_n}, we obtain
	$$
	|w_1(z)-w_0(z)|\le \int_z^\infty |\sin (t-z)||F(t)| |w_0(t)| |dt| \le M(y) \int_z^\infty |F(t)| |dt|, 
	$$
	which is \eqref{estimate w_n} for $ n=1 $. Suppose that \eqref{estimate w_n} is true for some $ n\ge 1 $. Then from \eqref{w_n},
	\begin{equation*}
	\begin{split}
	|w_{n+1} (z)-w_n(z)| &\le \int_z^\infty |\sin (t-z)||F(t)| |w_n(t)-w_{n-1}(t)| |dt|\\
	& \le M(y) \frac{1}{n!}\int_z^\infty |F(t)| \left[\int_t^\infty |F(u)| |du|\right]^n |dt|.
	\end{split}
	\end{equation*}
	By noticing that $ |F(t)| =- \frac{d}{|dt|} \int_t^\infty |F(u)| |du|$, we obtain 
	$$
	|F(t)| \left[\int_t^\infty |F(u)| |du|\right]^n = \frac{-1}{n+1}\frac{d}{|dt|}\left[\int_t^\infty |F(u)| |du|\right]^{n+1}.
	$$
	Thus
	$$
	|w_{n+1} (z)-w_n(z)| \le M(y) \frac{1}{(n+1)!} \left[\int_z^\infty |F(u)| |du|\right]^{n+1}.
	$$
	By induction, \eqref{estimate w_n} is true for all $ n\ge 1 $.  
	
	Third, we define a function  
	\begin{equation}\label{sol-w}
	w(z) = w_0(z) + \sum_{k=1}^{\infty} (w_{k}(z) - w_{k-1}(z)),\quad z\in D^+(\delta,R),
	\end{equation}
	and consider its properties. Using \eqref{estimate w_n} and \eqref{M}, 
	\begin{equation*}
	\begin{split}
	|w(z)| &\le  |w_0(z)|+\sum_{k=1}^\infty |w_k(z) - w_{k-1}(z)| \\
	&\le M(y)+M(y) \sum_{k=0}^\infty \frac{\alpha^k}{k!} \le  M(y)+M(y)e^\alpha, \quad z\in \Sigma_\sigma.
	\end{split}
	\end{equation*}
	This shows that $w(z)$ is bounded in $\Sigma_\sigma$, and hence it is analytic in $D^+(\delta,R)$. Moreover,
	\begin{eqnarray*}
		|w(z)-w_n(z)| &\le& \sum_{k=n+1}^\infty |w_k(z) - w_{k-1}(z)|
		\leq M(y) \sum_{k=n+1}^\infty \frac{\alpha^k}{k!}, 
		\quad z\in \Sigma_\sigma.
	\end{eqnarray*}
	Thus, the sequence $ \{w_n(z)\}$ convergences to $w(z)$ uniformly in $\Sigma_\sigma$, and it follows from \eqref{w_n} that $ w(z) $ satisfies the equation \eqref{vol-equation}. 
	
	By this proof, we have constructed one solution of  \eqref{vol-equation} in $D^+(\delta,R)$. 
	To prove that $w(z)$, constructed in \eqref{sol-w}, is the only solution of \eqref{vol-equation}, we assume that $W(z)$ is another solution of \eqref{vol-equation} in $D^+(\delta,R)$. Then
	\begin{equation*}
	|w(z)-W(z)| \le \int_{z}^\infty |F(t)|\; |w(t)-W(t)|\, |dt|,\quad z\in D^+(\delta,R),
	\end{equation*}
	where we have used the fact that $t=z+r$, \ $0\leq r<\infty $, and	 hence $|\sin(t-z)|\leq 1$.
	Using Lemma~\ref{theo1} with $ g = 0 $, we obtain $|w(z)-W(z)|\equiv 0$ in $D^+(\delta,R)$. This proves the uniqueness of $ w(z) $.
	
	Now, to show that $ w(z) $ is a solution of \eqref{1}, twofold differentiation of \eqref{vol-equation} gives
	\begin{eqnarray}
	w'(z) &=& w_{\sin}'(z) -\int_{z}^{\infty }\cos (t-z) F(t)w(t)\, dt, \nonumber\\
	w''(z) &=&  w_{\sin}''(z) +F(z)w(z)-\int_{z}^{\infty }\sin (t-z)F(t)w(t)\, dt. \label{V3}
	\end{eqnarray}		
	From \eqref{vol-equation} and \eqref{V3}, we obtain
	$$
	w''(z) + (1- F(z)) w(z) = w''_{\sin}(z) + w_{\sin}(z)=0. 
	$$ 
	Thus $ w(z) $ is a solution of equation \eqref{1}.
	
	It remains to prove \eqref{est-w}. Denoting $z=x+iy$, it follows from \eqref{vol-equation} that
	\begin{equation*}\label{start}
	\begin{split}
	\left| w(x+iy)- w_{\sin}(x+iy)\right|
	\le& \int_x^\infty |F(s+iy)||w(s+iy)|\, ds\\
	\le & \int_x^\infty |F(s+iy)| |w_{\sin}(s+iy)|\, ds\\
	& +\int_x^\infty |F(s+iy)||w(s+iy)-w_{\sin}(s+iy)|\, ds.
	\end{split}
	\end{equation*}
	To simplify the notation, set $w(s+iy)=w(s)$, $w_{\sin}(s+iy)=w_{\sin}(s)$ and $F(s+iy)=F(s)$. 
	Using Lemma \ref{theo1} with
	$$
	g(x)=\int_x^\infty |F(s)| |w_{\sin}(s)|\, ds,
	$$ 
	we obtain
	\begin{align*}
	\left| w(z)- w_{\sin}(z)\right|  \le & \ \int_x^\infty |F(s)| |w_{\sin}(s)|\, ds  \nonumber\\ 
	& +\int_x^\infty   |F(s)| \exp\left\{\int_x^s |F(u)|\, du \right\}\left( \int_s^\infty |F(u)| |w_{\sin}(u)|\, du \right)\, ds, \nonumber\\
	\le &\ M(y) \left[\int_x^\infty |F(s)|\, ds \right.  \nonumber\\
	&+ \left. \int_x^\infty |F(s)| \exp\left\{\int_x^s |F(u)|\, du \right\}\left( \int_s^\infty |F(u)|\, du \right)\, ds\right], \label{w-wsin}
	\end{align*}
	where $M(y)=\sup\limits_{s\in [x,\infty)} |w_{\sin}(s+iy)|$.	
	Set 	
	$$
	h(x) = \int_{x}^{\infty} |F(s)| ds.
	$$
	Then, by partial integration, we obtain
	\begin{align*}
	h(x)- e^{h(x)} \int_x^\infty h(s) h'(s) e^{-h(s)} ds &= h(x)-e^{h(x)} \Big(h(x)e^{-h(x)} -1 +e^{-h(x)}\Big) = e^{h(x)}-1,
	\end{align*}
	from which it follows that
	\begin{equation*}
	\left| w(z)- w_{\sin}(z)\right|\le M(y)\left[\exp\left\{\int_x^\infty |F(u)|\, du \right\}-1\right],
	\end{equation*}
	which is \eqref{est-w}.
\end{proof}

\medskip

A similar argument as in the proof of Theorem~\ref{volterra} will yield the following result.

\begin{theorem}\label{volterra2}
	Suppose that $ F(z) $ satisfies Hypothesis $ \mathrm{F}^- $, and let $ w_{\sin}(z) $ be a non-trivial solution of the sine equation
	\eqref{3}. Then the singular Volterra integral equation
	\eqref{vol-equation}, where $ z\in D^-(\delta,R) $ and the path of integration is $ t-z = -r $, $ 0\le r< \infty $, has a unique solution $ w(z) $ which is a solution of  equation \eqref{1}. Moreover, with $ z=x+iy $ we have
	\begin{equation}\label{est-w2}
	\left\vert w(z) -w_{\sin}(z) \right\vert \leq
	M(y) \left\{ \exp \left[ \int_{x}^{-\infty }\left\vert F(s+iy) \right\vert ds\right] -1\right\} , \quad z\in D^-(\delta,R),
	\end{equation}
	where $ M(y) = \sup\limits_{s\le x} |w_{\sin}(s+iy)| $.
\end{theorem}

\subsection{\sc Three types of asymptotic solutions}\label{asymp_sols}

The sine equation \eqref{3} has three types of nontrivial  solutions: 
	\begin{itemize}
		\item[(1)] Solutions of type $e^{iz}$, which decay
		to zero exponentially in the upper half-plane. 
		\item[(2)] Solutions of type $e^{-iz}$, which decay to zero exponentially in the lower half plane. 
		\item[(3)] Oscillatory solutions (i.e., solutions with infinitely many zeros) of type $\sin(z-z_0)$, which are nontrivial linear combinations of $e^{iz}$ and $e^{-iz}$.  
	\end{itemize}
Once a particular solution $ w_{\sin}(z) $ of \eqref{3} is chosen, from Lemma~\ref{volterra} it is natural to expect that
the corresponding solution $ w(z) $ of the perturbed sine equation \eqref{1} will inherit the asymptotic properties of $w_{\sin}(z)$. This subsection aims to show that this is indeed the case.

{
\begin{corollary}\label{t22} 
	Suppose that $F(z)$ satisfies Hypothesis $\mathrm{F}^+$.	Then the perturbed sine equation \eqref{1} has unique linearly independent nonoscillatory solutions $E^{+}(z)$ and $E^{-}(z)$ asymptotic to $ e^{iz} $ and $ e^{-iz} $, respectively, in $D^+(\delta,R)$ in the sense that
	\begin{equation}\label{asymp-sol}
	E^+(z) = e^{iz} \left(1+ v_1(z)\right) \quad \text{and} \quad E^-(z) = e^{-iz} \left(1+ v_2(z)\right),
	\end{equation}
where 
	\begin{equation}\label{R_s}
	|v_{s}(z)| \le \exp \left[
	\int_{z}^{\infty }\left\vert F(t)\right\vert \left\vert
	dt\right\vert \right] -1, \quad s=1,2,
	\end{equation}
and the path of integration is $ t-z = r $, $ 0\le r< \infty $, for 
each $z\in D^+(\delta,R)$. 
\end{corollary}


\begin{proof}
Obviously, $ w_{\sin}(z) = e^{iz} $ is a solution of  equation \eqref{3}. Then, from Theorem~\ref{volterra}, there exists a unique solution $ E^+(z) $ of equation \eqref{1} satisfying \eqref{est-w}, i.e., for $ z\in D^+(\delta,R) $, we~have
		\begin{equation}\label{we(iz)}
		|e^{iz}|\left\vert E^{+}(z)e^{-iz}-1\right\vert = \left\vert E^+(z) -e^{iz} \right\vert \leq
		M(y) \left\{ \exp \left[ \int_{z}^{\infty }\left\vert F(t) \right\vert |dt|\right] -1\right\} , 
		\end{equation}
	where the path of integration is $ t-z = r $, $ 0\le r< \infty $. Here we have 
		\begin{equation}\label{sup eiz}
		M(y) = \sup_{s\ge x} |e^{i(s+iy)}|  = e^{-y} = |e^{iz}|, \quad z=x+iy.
		\end{equation}
	From \eqref{we(iz)} and \eqref{sup eiz}, we see that $ v_1(z) = E^+(z)e^{-iz}-1 $ satisfies \eqref{R_s}.   Moreover,
	from \eqref{R_s} and \eqref{trimming-edges}, we infer that $|v_1(z)|<1$ for 
each $z\in D^+(\delta,R')$, where $ R'>R $ is large enough. Therefore, $E^+(z)$ has no zeros in $ D^+(\delta,R') $, and consequently $ E^+(z) $ is nonoscillatory in $ D^+(\delta,R) $ by the standard uniqueness theorem. A similar argument yields the conclusions for $ E^-(z) $.  From \eqref{asymp-sol},  $ E^+(z) $ and $ E^-(z) $ are clearly linearly independent in $ D^+(\delta, R) $.
	\end{proof}



By using Theorem~\ref{volterra2} and a similar proof to that of Corollary~\ref{t22}, 
we obtain the following result.
	\begin{corollary} \label{vol-}
		Suppose that $F(z)$ satisfies {Hypothesis $\mathrm{F}^-$}.	Then the perturbed sine equation \eqref{1} has 
		unique solutions $E^{+}(z)$ and $E^{-}(z)$ asymptotic to the respective functions $e^{ iz}$ and $e^{- iz}$ in $D^-(\delta,R)$ 
		in the sense that \eqref{asymp-sol} and \eqref{R_s} 
		hold, where the path of integration is $t-z= -r$, $ 0\le r<\infty $, for 
		each $z\in D^-(\delta,R)$.
	\end{corollary}
}

As discussed earlier, we proceed to consider solutions of the perturbed sine equation \eqref{1} which are related to oscillatory solutions of type $\sin(z-z_0)$ of the sine equation \eqref{3}.

\begin{corollary}\label{t2}
	For any oscillatory solution  $ S(z)$ of \eqref{1} in $ D^+(\delta,R) $, there exist two constants $ b\neq 0 $ and $ z_0=x_0+iy_0 $, such that
	\begin{equation}\label{osc-solution}
	S(z) = b\Big[\sin(z-z_0) + v(z)\Big],
	\end{equation}
	where
	\begin{equation}\label{osc-solution-2}
	|v(z)| \le  \cosh \left( y-y_{0}\right)
	\left\{ \exp \left[\int_{z}^{\infty }\left| F(t) \right| |dt| \right] -1 \right\},
	\quad  z\in D^+(\delta,R),
	\end{equation}
	and the path of integration is $ t-z = r $, $ 0\le r< \infty $. 
\end{corollary}

\begin{proof}
	From Corollary~\ref{t22},  $ E^+(z) $ and $ E^-(z) $ are linearly independent nonoscillatory solutions of \eqref{1} in $D^+(\delta,R)$. Let $ S(z) $ be any solution of \eqref{1} in $D^+(\delta,R) $. Then there exist two constants $ c_1 $ and $ c_2 $ such that
	\begin{equation}\label{new-1}
	S(z) = c_1 E^+(z) + c_2 E^-(z), \quad z\in D^+(\delta,R).
	\end{equation}
	If $ S(z) $ is oscillatory, then $ c_1c_2\neq0 $. Let $ z_0 $ be a point satisfying $ e^{-2iz_0} = -c_1 / c_2 $ and let $ b= 2ic_1 e^{iz_0} $. Then  \eqref{osc-solution} and \eqref{osc-solution-2} follow from \eqref{asymp-sol} and \eqref{new-1}.
\end{proof}

By using Corollary~\ref{vol-} and a similar proof to that of Corollary~\ref{t2}, we obtain the following result.
\begin{corollary}	\label{remark-ab}
	For any oscillatory solution $ S(z)$ of \eqref{1} in $ D^-(\delta,R) $, there exist two constants $ b\neq 0 $ and $ z_0=x_0+iy_0 $, such that \eqref{osc-solution} and \eqref{osc-solution-2} hold, where the path of integration is $ t-z = -r $, $ 0\le r< \infty $. 
\end{corollary}

\subsection{\sc  Zero distribution of oscillatory solutions}\label{zeros-asymptotic-solutions-sec}
Let $ S(z) $ be an oscillatory solution of \eqref{1} in $ D^+(\delta,R) $. Then from \eqref{osc-solution} we see that the zeros of $ S(z) $ are the zeros of $ \sin(z-~z_0) + v(z) $. Note that $\sin(z-z_0)$ is oscillatory on the horizontal line $\im (z)=y_0$. We will prove that $S(z)$ is oscillatory in the intersection of the horizontal strip $y_0-\gamma<\im (z)<y_0+\gamma$ and the domain $D^+(\delta,R)$, where $\gamma>0$ is a small constant. 
%
Let $H_0$ denote a half-plane
	\begin{equation*}\label{H_0}
	H_0=\{z:\re(z)>\sigma_{0}\},
	\end{equation*}
where $\sigma_0>0$ is chosen large enough so that both $H_0\subset D^+(\delta,R)$ and   
	\begin{equation}\label{34}
	\exp \left[\int_{z}^{\infty }\left| F(t) \right| |dt| \right] 
	<1+\frac{\sin (\gamma) }{\cosh (\gamma) },\quad z\in H_0,  
	\end{equation}
are satisfied. Observe that \eqref{34} follows from \eqref{trimming-edges}. In addition,  we may assume that $z_0=x_{0}+iy_{0}$ satisfies $z_{0}-\gamma \in
H_0$ and $z_0-\pi +\gamma\not\in H_0$. For $k\geq 0$, let 
$Q_{k,\gamma }$ denote the square
	\begin{equation*}\label{squares}
	Q_{k,\gamma }=\{z=x+iy :\left\vert x-x_{0}-k\pi \right\vert <\gamma,\ \left\vert
	y-y_{0}\right\vert <\gamma\}.
	\end{equation*}
For any fixed $k\geq 0$, the point $z_0+k\pi$ is the center of the square $Q_{k,\gamma}$, see Figure \ref{f1}. 
\begin{figure}[ht]
		\definecolor{uququq}{rgb}{0.25,0.25,0.25}
		\begin{tikzpicture}[line cap=round,line join=round,>=triangle 45,x=1.0cm,y=1.0cm]
		\clip(-5.3,-3.18) rectangle (5.46,3.14);
		\draw (-2,2)-- (-2,-2);
		\draw (-2,-2)-- (2,-2);
		\draw (2,-2)-- (2,2);
		\draw (-2,2)-- (2,2);
		\draw [dash pattern=on 1pt off 1pt] (-3,2)-- (-2,2);
		\draw [dash pattern=on 1pt off 1pt] (-3,-2)-- (-2,-2);
		\draw [dash pattern=on 1pt off 1pt] (2,2)-- (3,2);
		\draw [dash pattern=on 1pt off 1pt] (2,-2)-- (3,-2);
		\draw (2.,0.4) node[anchor=north west] {\tiny $L_1$};
		\draw (-0.32,2.5) node[anchor=north west] {\tiny $L_2$};
		\draw (-2.7,0.4) node[anchor=north west] {\tiny $L_3$};
		\draw (-0.32,-1.98) node[anchor=north west] {\tiny $L_4$};
		\draw (-0.2,0.05) node[anchor=north west] {\tiny $z_0+k\pi$};
		\draw (1.3,2.7) node[anchor=north west] {\tiny $(x_0+k\pi+\gamma)+i(y_0+\gamma)$};
		\draw (-5,2.7) node[anchor=north west] {\tiny $(x_0+k\pi-\gamma)+i(y_0+\gamma)$};
		\draw (-4.9,-2.08) node[anchor=north west] {\tiny $(x_0+k\pi-\gamma)+i(y_0-\gamma)$};
		\draw (1.3,-2.08) node[anchor=north west] {\tiny $(x_0+k\pi+\gamma)+i(y_0-\gamma)$};
		\begin{scriptsize}
		\fill [color=uququq] (0,0) circle (1.5pt);
		\fill [color=black] (2,-2) circle (1.5pt);
		\fill [color=black] (2,2) circle (1.5pt);
		\fill [color=black] (-2,2) circle (1.5pt);
		\fill [color=black] (-2,-2) circle (1.5pt);
		\end{scriptsize}
		\end{tikzpicture}
		\caption{The square $Q_{k,\gamma}$.}
		\label{f1}
	\end{figure}
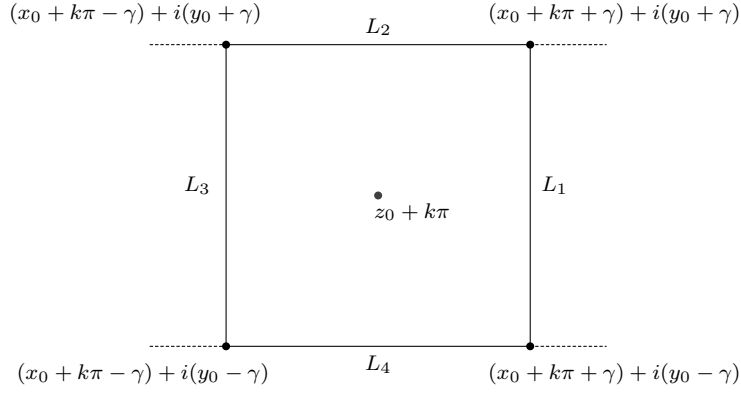

\begin{lemma}\label{t5}
The function $S(z)$ in \eqref{osc-solution} is oscillatory in the half-plane $H_0$. Specifically, $S(z)$ has precisely one zero in each square $Q_{k,\gamma }$, and no other zeros in $H_0$. In addition, we have
	\begin{equation}\label{nr}
	n\left(r,H_0, \frac{1}{S}\right)= \frac{r}{\pi}  (1+o(1)), \quad r\to\infty,
	\end{equation}
where $ n(r,H_0, 1/S) $ counts only those zeros  of $ S(z) $ that lie in $ H_0 $ and $ |z|\le r $.
 \end{lemma}

\begin{proof}
Without loss of generality, we may assume that $ b=1 $ in \eqref{osc-solution}.
	Let $z\in\partial Q_{k,\gamma}$. We have
	\begin{eqnarray*}
		|\sin(z-z_0)|^2 &=& \frac14 \left(e^{2(y-y_0)}+e^{-2(y-y_0)}+2\left(\sin^2(x-x_0)-\cos^2(x-x_0)\right)\right)\\	
		& =&\frac14 \left(e^{2(y-y_0)}-2+e^{-2(y-y_0)}+4\sin^2(x-x_0)\right)\\
		&=& \sinh^2(y-y_0)+\sin^2(x-x_0).
	\end{eqnarray*}
	\begin{itemize}
		\item[(1)]
		On the line segments $L_1$ and $L_3$, we have $x-x_0=k\pi \pm \gamma$, $|y-y_0|\le \gamma$, and hence 
		\[
		|\sin(z-z_0)|^2= \sin^2(\gamma)+\sinh^2(y-y_0)\ge \sin^2(\gamma),
		\]
		i.e., $|\sin(z-z_0)|\ge \sin(\gamma)$. On the other hand, from \eqref{osc-solution-2} and \eqref{34}, we obtain
		\begin{eqnarray*}
			\left\vert v\left( z\right) \right\vert &\leq& \cosh \left( y-y_{0}\right)
			\left\{ \exp \left[ \int_{0}^{\infty }\left\vert F\left( z+s\right)
			\right\vert ds\right] -1\right\} \\
			&\le & \cosh(\gamma) \left\{ \exp \left[ \int_{0}^{\infty }\left\vert F\left( z+s\right)
			\right\vert ds\right] -1\right\}<\sin(\gamma).
		\end{eqnarray*}
		This gives us $|\sin(z-z_0)|>|v(z)|$ for $z\in L_1 \cup L_3$.
		\item[(2)]
		On the line segments $L_2$ and $L_4$, we have $k\pi-\gamma \le x-x_0\le k\pi+\gamma$, $y-y_0=\pm \gamma$, and hence
		\[
		|\sin(z-z_0)|^2= \sin^2(x-x_0)+\sinh^2(\gamma)\ge \sinh^2(\gamma)\ge \sin^2(\gamma),
		\]
		i.e., $|\sin(z-z_0)|\ge \sin(\gamma)$. On the other hand, similarly as in Case~(1), 
		$|v(z)|<\sin(\gamma)$.  
		This gives us $|\sin(z-z_0)|>|v(z)|$ for $z\in L_2 \cup L_4$.
	\end{itemize}
	From Cases~(1) and (2), we obtain $|\sin(z-z_0)|>|v(z)|$ on $\partial Q_{k,\gamma}$. As the function $\sin(z-z_0)$ has precisely one zero in each square $Q_{k,\gamma}$, namely $z_0+k\pi$, then by Rouch\'e's theorem, the solution $S(z)=\sin(z-z_0)+ v(z)$ has also precisely one zero in each square $Q_{k,\gamma}$. Obviously, $S(z)$ has no zeros on $\partial Q_{k,\gamma}$.
	
	Now we show that $S(z)$ has no zeros in $H_0$ that lie outside the closed squares $Q_{k,\gamma}\cup\partial Q_{k,\gamma}$.
	If $|y-y_0|>\gamma$, then 
	\begin{eqnarray*}
		|\sin(z-z_0)|^2 &=& \sin^2(x-x_0)+\sinh^2(y-y_0)\ge \sinh^2(y-y_0)\\
		&=& \tanh^2(y-y_0)\cosh^2(y-y_0)\geq \tanh^2(\gamma)\cosh^2(y-y_0),
	\end{eqnarray*}
	while from \eqref{osc-solution-2} and \eqref{34}, we have
	\begin{equation*}
	\left\vert v(z) \right\vert <\cosh (y-y_{0}) \frac{\sin (\gamma)}{\cosh (\gamma)}
	<\cosh (y-y_{0}) \tanh (\gamma).
	\end{equation*}
	It follows that $|\sin (z-z_{0})|>|v(z)|$, and hence $S(z) $ cannot vanish if $|y-y_0|>\gamma$.
	Finally, we consider the rectangles between the closed squares, i.e., the regions
	\begin{equation*}
	\gamma< x-x_{0}-k\pi <\pi -\gamma,\quad |y-y_{0}|\leq \gamma.
	\end{equation*}
	We have now $|\sin(z-z_{0})|>\sin (\gamma)$ and $|v(z)|<\sin(\gamma)$, so that $S(z)$ has no 	zeros in these regions either.
	
	It remains to estimate the number of zeros of $ S(z) $ in $ H_0 $. Suppose first that $ y_0-\gamma\ge 0 $. For each $ k\ge 0 $, $S(z)$ has exactly one zero in each square $Q_{k,\gamma}$ and no other zeros in $ H_0 $.
	%
	%
	Then if $ r $ satisfies 
	$$
	\sqrt{\left(x_{0}+k \pi-\gamma\right)^{2}+\left(y_{0}-\gamma\right)^{2}} \le r <\sqrt{\left(x_{0}+(k+1) \pi-\gamma\right)^{2}+\left(y_{0}-\gamma\right)^{2}}, \quad k\ge  0,
	$$
it can be observed that $ n(r,H_0, 1/S)= k  $ or $ k+1 $. Thus, 
		$$
		k= \frac{r}{\pi} (1+o(1)), \quad r\to \infty,
		$$
which gives \eqref{nr} in the case when $ y_0-\gamma\ge 0 $. An almost identical argument applies in the case when $ y_0+\gamma\le 0 $. Finally, when $ -\gamma<y_0<\gamma $, it can be shown that \eqref{nr} holds by considering
	$$
	x_{0}+k \pi-\gamma\le r< x_{0}+(k+1) \pi-\gamma, \quad k\ge0.
	$$
This completes the proof for any $ y_0 $.
\end{proof}


The reasoning in the proof of Lemma~\ref{t5} can easily be modified to prove the following result, where $ H_0^* $ and $ Q_{k,\gamma}^* $ are the reflections of $ H_0 $ and $ Q_{k,\gamma} $ with respect to the imaginary axis.

\begin{lemma}\label{oscillation-of-solutions}
 The function $S(z)$ in \eqref{osc-solution} is oscillatory in the half-plane $ H_0^* $. Specifically, $S(z)$ has precisely one zero in each square $Q_{k,\gamma }^*$, and no other zeros in $H_0^*$. In addition, we have 
 	\begin{equation*}\label{nr2}
 	n\left(r,H_0^*, \frac{1}{S}\right)= \frac{r}{\pi}  (1+o(1)), \quad r\to\infty,
 	\end{equation*}
 where $ n(r,H_0^*, 1/S) $ counts only those zeros  of $ S(z) $ that lie in $ H_0^* $ and $ |z|\le r $.
\end{lemma}



\section{Proof of Theorem \ref{theo-growth}}\label{proof-of-thm1}



Let $f(z)$ be a non-trivial solution of \eqref{e1}. Then from the transformation \eqref{g}, $g(z)$ is a solution of the normalized equation \eqref{s-e1}. Let $g_j(z)$ denote the restriction of $g(z)$ to $G_j(R)$ in \eqref{G_j}, for $ j\in\{0,\ldots,n+1\} $. By Lemma~\ref{Q-P} and \eqref{LT}, there is a solution $w_j(\z)$ of the perturbed sine equation \eqref{e2} in a domain containing $ \tilde{G}_j(\delta, \tilde{R}) $ in \eqref{image G_j}, such that
	\begin{equation}\label{pr121}
	g_j(z) =z^{-n/4} (1+\ell(z))^{-1/4} w_j(\zeta),\quad \z=L_j(z).
	\end{equation}
Regarding the fourth roots in \eqref{pr121}, $ z^{-n/4} $ is defined by \eqref{FR} and $ (1+\ell(z))^{-1/4} $ is defined by \eqref{PB}. For the rest of these Sections~\ref{proof-of-thm1}, \ref{proof-of-thm2} and \ref{proof-of-thm3}, the branches for square roots and fourth roots are defined by \eqref{PB},  \eqref{SR} and \eqref{FR}. 
In \eqref{pr121}, there is a unique correspondence between $g_j$ and $w_j$ by means of Lemma~\ref{L3-new}.

By Lemmas~\ref{L3-new} and \ref{Q-P}, the function $T(\zeta )$ in \eqref{e2} is analytic in a domain containing $\widetilde{G}_j(\delta, \tilde{R})$ for sufficiently small $ \delta $ and sufficiently large $ \tilde{R} $.
 Notice that $ T(\zeta) $ satisfies Hypothesis~$ \mathrm{F}^+ $ when $ j $ is even, and satisfies Hypothesis~$ \mathrm{F}^- $  when $ j $ is odd. Moreover, from Lemma~\ref{Q-P}, we obtain
	$$
	\int_\zeta^\infty |T(t)| |dt| = O\left(\frac{1}{|\zeta|}\right), \quad \zeta \to\infty,
	$$
where the path of integration is indicated as in Hypotheses $ \mathrm{F}^+ $ or $ \mathrm{F}^- $. 
 It follows from the results on asymptotic integration in Section~\ref{asymp_sols}, that equation \eqref{e2} has two linearly independent solutions $ w_j^+(\zeta) $ and $ w_j^-(\zeta) $ asymptotic to $ e^{i\zeta} $ and $ e^{-i\zeta} $, respectively, in $ \widetilde{G}_j(\delta,\tilde{R}) $. Then there exist two constants $c_{j1},c_{j2}$, such that
	\begin{equation}\label{w_j}
	w_j(\zeta)= c_{j1}w_j^+(\zeta)+c_{j2}w_j^-(\zeta)= \left( c_{j1}e^{i\zeta }+c_{j2}e^{-i\zeta } \right) (1+o(1)),
	\end{equation} 
as $\z\to\infty$ in $\widetilde{G}_j(\delta,\tilde{R})$.
From Lemma \ref{L2}, we have
		$$
		\z(z)=\frac{2}{n+2}z^{(n+2)/2}(1+o(1)),
		$$
as $ z\to\infty $ in $G_j(R)$. By substituting $\z(z)$ in $w_j(\z)$ and using \eqref{pr121}, we find that $g_j$ has the asymptotic form
		\begin{equation}\label{fasy}
		g_j(z)= z^{-n/4} \left(c_{j1}F_1(z)+c_{j2}F_2(z)\right) (1+o(1)),
		\end{equation}		
as $ z\to \infty $ in $ G_j(R) $, where
		$$
		F_s(z)= \exp\left\{ (-1)^{s+1} \frac{2i}{n+2} z^{(n+2)/2} (1+o(1))\right\},\quad s=1,2.
		$$

Set $z=re^{i\theta}$, and let $h_s(\theta)$ be the Phragm\'en-Lindel\"of indicator function of $F_s(z)$. Then~we~have
		\begin{align*}
		h_s(\theta)&= \limsup_{r\to\infty} r^{-{(n+2)/2}} \log |F_s(re^{i\theta})|\\
					&= \limsup_{r\to\infty} r^{-{(n+2)/2}} \log \left| \exp\left\{ (-1)^{s+1} \frac{2i}{n+2} r^{(n+2)/2} e^{i\frac{n+2}{2}\theta} (1+o(1))\right\}\right|\\
					&= (-1)^{s} \frac{2}{n+2} \sin\left( \frac{n+2}{2}\theta\right), \quad \psi_{j-1} <\theta<\psi_{j+1}, \quad s=1,2.
		\end{align*}
Hence, $h_1(\theta)=-h_2(\theta)$ for $ \psi_{j-1} <\theta<\psi_{j+1} $, and $ h_1(\psi_j)=h_2(\psi_j)=0 $.
%
This means that if $F_1(z)$ blows up (resp.~decays to zero) exponentially on all rays in an open sector between two consecutive critical rays, then $F_2(z)$ decays to zero (resp.~blows up) exponentially on all rays in the same sector. 

Set, for $ j=0, \ldots, n+1 $,
	\begin{equation}\label{plus-minus}
	\begin{split}
	G_{j}^+ &=\{z\in G_j(R): \psi_{j}<\arg (z)<\psi_{j+1}\},\\
	G_{j}^- &=\{z\in G_j(R): \psi_{j-1} < \arg (z)<\psi_j\}.
	\end{split}		
	\end{equation}
%
The restriction $g_j$, $j=0,\ldots,n+1$, cannot decay to zero exponentially in both $G_{j}^+$ and $G_{j}^-$, for otherwise $h_{g_j}(\theta)<0$ and $h_{g_j}(\theta+2\pi/(n+2))<0$, where $\theta\in (\psi_{j-1},\psi_{j})$ and $\theta+2\pi/(n+2) \in (\psi_{j},\psi_{j+1})$, 
which is impossible by \cite[p.~56]{levin}. This proves that assertion (b) in Theorem~\ref{theo-growth} holds for the solution $ g $ of the normalized equation \eqref{s-e1}. Hence, from \eqref{g}, the assertion (b) also holds for $ f $.

From (b) and the asymptotic form \eqref{fasy}, it follows that $g_j$ satisfies one of the following three possibilities:
\begin{enumerate}
\item[(i)] $g_j$ blows up in $G_{j}^+$ and decays to zero in $G_{j}^-$;

\item[(ii)] $g_j$ decays to zero in $G_{j}^+$ and blows up in $G_{j}^-$;

\item[(iii)] $g_j$ blows up in both $G_{j}^+$ and $G_{j}^-$.
\end{enumerate}
If (i) or (ii) holds, then either $c_{j1}=0$ or $c_{j2}=0$, say $c_{j2}=0$, and consequently $g_j$ is asymptotically comparable to the remaining exponential factor $F_1(z)$. 
If (iii) holds, then $c_{j1}c_{j2}\neq 0$. Since $h_1(\theta)=-h_2(\theta)$ for all $\theta\in (\psi_{j-1},\psi_{j+1})$, it follows that in each region $ G_j^+$, $G_j^-  $, precisely one of the two exponential terms is dominant. This means that, in one region, say $ G_j^+ $, we have $|g_j| \asymp |F_1(z)|$, while in the other one, which is $ G_j^- $, we have $|g_j| \asymp |F_2(z)|$, as $z\to\infty$.

The reasoning above can be repeated for the restrictions of $g$ in all the remaining sector 
pairs~$G_j(R)$. The union of $G_j(R)$'s includes the entire complex plane minus the disc $|z|\leq R$. Since the solutions of \eqref{s-e1} are entire, this gives a unique behaviour of $g$ in sectors. In particular, $g$ is asymptotically comparable to one of $F_s(z)$, $s=1,2$, in each sector between two adjacent critical rays. This proves that the first statement of part~(a) in Theorem~\ref{theo-growth} holds for $ g $.

Finally, from \eqref{g} we see that the above asymptotic behaviour of $ g $ holds for $ f $ along any ray between the critical rays  $ \arg( z)=\theta_j $ , $ j=0,\ldots, n+1 $, given in \eqref{sectors+rays}. Furthermore,  from \eqref{fasy}, along any ray between critical rays, $ f $ is asymptotically comparable to either $ E_1(z) $ or $ E_2(z) $, where  
	$$
	E_s(z) = F_s\left(\frac{z}{\mu}\right)=\exp\left\{ (-1)^{s+1} i d z^{q} (1+o(1))\right\},\quad s=1,2,
	$$
where $ \mu^{n+2}=p_n^{-1} $ and $ d $, $ q $ are defined in \eqref{consts}.  This proves that the first statement in part (a) holds for $ f $.

Recall that the domains $ G_j(R) $ in the $ z $-plane and the domains $ \widetilde{G}_j(\delta, \tilde{R}) $ in the $ \zeta $- plane are related by means of the transformations $ \zeta =L_j(z)$, see \eqref{LT} and  Lemma~\ref{L3-new}. As discussed earlier in this section, there are two solutions $ w_j^+(\zeta) $ and $ w_j^-(\zeta) $ of \eqref{e2} in the domain $ \widetilde{G}_j(\delta, \tilde{R}) $, which are asymptotic to $ e^{i\zeta} $ and $ e^{-i\zeta} $. Define the functions $ g_j^+(z) $ and $ g_j^-(z) $ as
	$$
	g_j^+(z) = z^{-n/4} (1+\ell(z))^{-1/4} w_j^+(\zeta)\quad \text{and} \quad  g_j^-(z) = z^{-n/4} (1+\ell(z))^{-1/4} w_j^-(\zeta),
	$$
where $ \zeta= L_j(z)$. 
Since $ w_j^+(\zeta) $ and $ w_j^-(\zeta) $ are solutions of equation \eqref{e2} with $ T(\zeta) $ given in \eqref{T},  it follows 
that $ g_j^+(z) $ and $ g_j^-(z) $ are solutions of equation \eqref{s-e1}. 
	We show this for $ g_j^+(z) $ only. From the definition of $ g_j^+(z) $, we see that $ 	w_j^+(\zeta) = B(z)g_j^+(z)$, where $ B(z)= z^{n/4} (1+\ell(z))^{1/4} $. Using the same calculations from the proof of Lemma~\ref{Q-P} down to \eqref{proof1}, we obtain
	\begin{equation}\label{ew-1}
	(w^+_j)''(\zeta )=\frac{1}{4}\frac{1}{A(z)}g_j^+(z)\frac{d}{dz}\left( 
		\frac{Q^{\prime }(z)}{B(z)^{3}A(z)}\right) +\frac{B(z)}{Q(z)}(g_j^+)''(z),
	\end{equation}
	where $ A(z)= z^{n/2} (1+\ell(z))^{1/2} $. Since $ w_j^+(\zeta) $ is a solution of \eqref{e2} with $ T(\zeta) $ given in \eqref{T}, it follows that  
		\begin{equation}\label{ew}
		(w^+_j)''(\zeta ) = (T(\zeta)-1) w^+_j(\zeta ).
		\end{equation}
	By replacing $ T(\zeta) $ with \eqref{proof3}  and by replacing  $ w^+_j(\zeta ) $ with $ B(z) g_j^+(z) $ in \eqref{ew}, we get
	 \begin{equation}\label{ew+1}
	 	(w^+_j)''(\zeta ) =\frac{1}{4}\frac{1}{A(z)}g^+_j(z)\frac{d}{dz}\left(
	 \frac{Q^{\prime }(z)}{B(z)^{3}A(z)}\right) -B(z)g^+_j(z).
	 \end{equation}
	 Now from \eqref{ew-1} and \eqref{ew+1}, we obtain that $ g_j^+(z) $ is a solution \eqref{s-e1}.

By substituting $ \zeta(z) $ in $ w_j^+(\zeta) $ and in $ w_j^-(\zeta) $, we obtain that $ g_j^+(z) $  and $ g_j^-(z) $ are asymptotically comparable to $ F_1(z) $ and $ F_2(z) $, respectively. Therefore, for any given sector between critical rays, there are two solutions of \eqref{s-e1} corresponding to $ F_1(z) $ and $ F_2(z) $. Finally, from \eqref{g} we deduce that \eqref{e1} has two solutions corresponding to $ E_1(z) $ and $ E_2(z) $. This proves the last sentence in part (a).


\section{Proof of Theorem \ref{theo-zeros}}\label{proof-of-thm2}
For $ j=0, \ldots,n+1 $, let $ f(z) $, $ g(z) $, $ g_j(z) $ and $ w_j(\zeta) $ be as in the proof of Theorem \ref{theo-growth}. In particular, $ w_j(\zeta) $ is a solution of \eqref{e2}, where $ T(\zeta) $ satisfies Hypothesis~$\mathrm{F}^+ $ when $ j $ is even and Hypothesis~$ \mathrm{F}^- $  when $ j $ is odd.

Suppose that $ j $ is even, and suppose
that $w_j(\zeta)$ is oscillatory in $\tilde{G}_j(\delta,\tilde{R})$. Then it follows from 
 Lemma~\ref{t5} that the zeros of $w_j(\zeta)$ in $H_0$ lie, for some $ v_0\in\R $, in the squares $Q_{k,\gamma}$ in-between the horizontal half-lines
$$
\ell_{j1}:\z= u+i(v_0-\gamma)\quad
\textnormal{and} \quad
\ell_{j2}:\z= u+i(v_0+\gamma),
$$ 	
where $\gamma>0$ is small and $u>0$. Moreover, $w_j(\zeta)$ has no other zeros in $H_0$. 
All the zeros of $w_j(\zeta)$, apart from finitely many possible exceptions, lie in the squares $Q_{k,\gamma}$.  
		From Lemma~\ref{L1}, we see that the pre-image of any horizontal half-strip bounded by $ \ell_{j1} $ and $ \ell_{j2} $ 
		under the mapping $\zeta=L_j(z)$, $z\in G_j(R)$, is contained in a domain  
			\begin{equation}\label{K}
			\Omega_{j}=\{z=re^{i\theta} : |\theta-\psi_j|<Ce_n(r),\,r>R\},
			\end{equation}
		where $C>0$ depends on $n,R$. 
Since the zeros of $g_j$ in ${G}_j(R)$ depend on the zeros of $w_j(\z)$ in $\tilde{G}_j(\delta,\tilde{R})$ in a one-to-one manner by Lemma~\ref{L3-new}, it follows that the zeros of $g_j$ are located in the curvilinear strip $\Omega_{j}$ 
around the ray $ \arg( z) = \psi_{j} $, with finitely many possible exceptions. We may repeat the above process for every $j$, and find that all zeros of $g$, apart from finitely many possible exceptions, lie in the union $\bigcup_{j=0}^{n+1} \Omega_{j}$. 
Now, from \eqref{g}, it is easy to deduce that all but at most finitely many zeros of $ f $ lie in the union $\bigcup_{j=0}^{n+1} \Lambda_{j,c}$, where the $ \Lambda_{j,c} $ are the translates \eqref{CS} of the domains 
	$$
	\Lambda_{j}=\{z=re^{i\theta} :|\theta-\theta_j|<Ce_n(r),\,r>R\},
	$$
which are the rotation of the domains \eqref{K}.

It remains to estimate the number of zeros of  $ f $ in $\Lambda_{j,c}$. Let us first estimate the number of zeros of $ g $ in $\Omega_{j}$
, i.e., the number of zeros of $ g_j $. From Lemma~\ref{t5}, we have 
	$$
	n(\rho,H_0, 1/w_j)= \frac{\rho}{\pi} (1+o(1)), \quad \rho\to\infty.
	$$
Since $ w_j(\zeta) $ can have at most finitely many zeros that do not lie  in the squares, we obtain
	$$
	n(\rho,1/w_j)= \frac{\rho}{\pi} (1+o(1)), \quad \rho\to\infty.
	$$
By Lemma~\ref{L2}, we have 
	$$
	\rho= \frac{2}{n+2}\; r^{(n+2)/2} (1+o(1)),\quad r\to\infty.
	$$
Hence, the number of zeros of $g$ with modulus $ \le r $ in $\Lambda_{j}^0$
, where $r>R$,  is given by
	\begin{equation*}
	n(r,\Omega_j,1/g)= \frac{2}{(n+2) \pi} \; r^{(n+2)/2} (1+o(1)),\quad r\to\infty. 
	\end{equation*}
Now, from \eqref{g}, it is easy to obtain that the number of zeros of $f$ with modulus $ \le r $ in $\Lambda_{j,c}$, where $r>R$, is given by
	\begin{equation*}
	n(r,\Lambda_{j,c},1/f) = \frac{2\sqrt{|p_n|}}{(n+2) \pi} \; r^{(n+2)/2} (1+o(1)),\quad r\to\infty. 
	\end{equation*}
This completes the proof of \eqref{n}, and \eqref{N} follows by a simple integration.

If $ j $ is odd, then we get the results by using similar reasoning with Lemma~\ref{oscillation-of-solutions} instead of Lemma~\ref{t5}.


\section{Proof of Theorem \ref{theo-growth-zero}}\label{proof-of-thm3}

(a) Fix the ray $\arg(z)=\theta_j$, and suppose that $f$ blows up exponentially on each ray in both sectors 
$S(\theta_{j-1},\theta_j)$ and $S(\theta_j,\theta_{j+1})$. Then from \eqref{g}, $ g $ blows up  exponentially on each ray of the sectors $S(\psi_{j-1},\psi_j)$ and $S(\psi_j,\psi_{j+1})$, i.e., in the domains $ G_j^+ $ and $ G_j^- $ defined in \eqref{plus-minus}. From the proof of Theorem~\ref{theo-growth}, 
$g$ has the asymptotic form 
\eqref{fasy} in $G_j(R)$ with $c_{j1}c_{j2}\neq 0$, and precisely one of the exponential terms $F_1(z), F_2(z) $ is dominant in each of the sectors $G_{j}^-$ and $G_j^+$. Without loss of generality, we may suppose that $F_1(z)$ is dominant in $G_{j}^-$. Then $F_2(z)$ decays to zero in $G_{j}^-$, and the roles of $F_1(z)$ and $F_2(z)$ are interchanged in $G_j^+$. Thus we may re-write \eqref{fasy} in the form
	$$
	g(z)=\left\{\begin{array}{rl}
	c_{j1}z^{-n/4} \left(F_1(z)+F_2(z)\right) (1+o(1)),\ & z\in G_{j}^-,\\
	c_{j2}z^{-n/4} \left(F_1(z)+F_2(z)\right) (1+o(1)),\ & z\in G_j^+.
	\end{array}\right.
	$$
Hence, $g$ is asymptotically comparable to
	$
	\sin (z^{(n+2)/2}(1+o(1)))
	$
as $z\to\infty$ in either of the sectors $G_{j}^-$ or $G_j^+$. Since $g$ is entire, this asymptotic form holds in $G_j(R)$, for sufficiently large~$ R $. 

Since the function $\sin(z^{(n+2)/2})$ is oscillatory on $\arg(z)=\psi_j$, it follows that $g$ is oscillatory near $\arg(z)=\psi_j$, and then from \eqref{g}, $ f $  is oscillatory near $\arg(z+c)=\theta_j$. More precisely, from Theorem~\ref{theo-zeros}, we find that all but at most finitely many zeros of $f$ lie in the $\Lambda_{j,c}$ that encloses the ray $\arg(z+c)=\theta_j$. The domain $\Lambda_{j,c}$ in turn is essentially contained in the $\varepsilon$-sector $W_j(\varepsilon)$. Thus the ray $\arg (z)=\theta_j$ is non-shortage by definition.

(b) Suppose that $f$ decays to zero exponentially on each ray in $S(\theta_j,\theta_{j+1})$. From Theorem~\ref{theo-growth}(b), we find that $f$ blows up exponentially on each ray in both sectors $S(\theta_{j+1},\theta_{j+2})$ and $S(\theta_{j-1},\theta_{j})$.  From the proof of Theorem~\ref{theo-growth}, we see that from all three possible combinations of $ E_1(z) $ and $ E_2(z) $, the solution $ f $ must be asymptotically comparable to one of $E_1(z),  E_2(z) $ along each ray in $W_j(\varepsilon)$, and asymptotically comparable to the other one of $E_1(z),  E_2(z) $ along each ray in $W_{j+1}(\varepsilon)$. 

Since the exponential terms $E_s(z)$ are zero-free functions, this means that apart from a disc of large radius $R$, $f$ cannot vanish in $W_j(\varepsilon)$ and $W_{j+1}(\varepsilon)$. Thus the rays  $\arg (z)=\theta_j$ and $\arg (z)=\theta_{j+1}$ are shortage by definition.


\appendix
\section{Branch cuts and branches}\label{appendix}

In \cite[Ch.~7.4]{hille}, Hille considers multi-valued functions on Riemann surfaces, which can sometimes seem abstract and difficult to visualize, especially for non-expert readers. Alternatively,  branch cuts can be made in the complex plane  so that the functions in question are single-valued and analytic everywhere except on these branch cuts.  In the literature,  either the real negative axis or the real positive axis is typically chosen as the branch cut for elementary multi-valued functions such as  square roots. However, when operating in various
sectors as we do with Liouville's transformation, these two branch cuts do not always work. 
When using multiple branch cuts and branches
in the same reasoning leads to questions about  how they are related to one another.
 In this appendix, we will discuss this relationship
		for the square root. The branches for the fourth root  can be handled in a similar way.

Given a real number $ \vp $, let $\mathcal{C}_\vp$ be the branch cut along the ray $\arg (z)=\vp$. For any $ m\in\mathbb{Z} $, 
the symbol $ \mathcal{C}_{\vp+2m\pi} $ represents the same cut. Each branch cut is associated with two branches of the square root, where one branch is defined by
	$$
	{z}^{1/2}= \sqrt{|z|} \exp\left(i \frac{\arg (z)}{2}\right), \quad \vp<\arg (z) \le \vp + 2\pi,
	$$
and the second branch is for $\vp+2\pi<\arg (z) \le \vp + 4\pi$.
Since the first branch is related to the arguments on the interval $(\vp, \vp+2\pi]$, the branch is denoted by $_\varphi\sqrt{\cdot}=(\cdot)^{1/2}$, and the argument on that interval is denoted by $\arg_\varphi(\cdot)$. So, we can write the branch as follows:
	\begin{equation}\label{square root}
	_\varphi\sqrt{z}= \sqrt{|z|} \exp\left(i \frac{\arg_\varphi (z)}{2}\right).
	\end{equation}
In addition, it is easy to see that 
	$$
	_\vp\sqrt{\cdot} \equiv {_{\vp+4m\pi}}\sqrt{\cdot} \equiv - \big({_{\vp+2\pi}}\sqrt{\cdot} \, \big)
	$$
for any real number $\vp$ and any integer $m$. In particular, we may consider $ \s{\cdot}{\vp} $ and $ \s{\cdot}{\vp+2\pi} $ as two different branches of the square root related to the branch cut $ \mathcal{C}_{\vp} $.

{The following lemma shows, in Theorem~\ref{theo-growth}, that any fixed branch cut outside the sector $S$ and any fixed branch for the square roots  can be chosen in the functions $E_1(z)$ and $E_2(z)$, provided that they are the same for both $E_1(z)$ and $E_2(z)$.}

\begin{lemma}\label{different-branches}
Let $\varphi,\psi\in\R$, and let $z\in\C\setminus\{0\}$. Then either $\s{z}{\vp}=\s{z}{\psi}$
or $\s{z}{\vp}=-\s{z}{\psi}$.
\end{lemma}

Before proving Lemma~\ref{different-branches}, we will find an alternative representation for the function
$\s{\cdot}{\vp}$, $\varphi\in\R$.
Let  $z=re^{i\theta} \in \C$ be a point with a fixed argument $\theta\in\mathbb{R}$ such that $ \theta \not= \varphi+ 2\pi \mathbb{Z}$. Then $\arg_\vp (z)= \theta +2k\pi$ for some unique integer~$k = k(\theta,\vp)$~satisfying
	$
	\vp < \theta +2k\pi \le \vp +2\pi,
	$
or
	$$
	\frac{\vp-\theta}{2\pi} < k \le \frac{\vp-\theta}{2\pi} +1.
	$$
We have 
	\begin{equation*}
		\left\lceil\frac{\vp-\theta}{2\pi}\right\rceil \le k \le \left\lfloor\frac{\vp-\theta}{2\pi} +1 \right\rfloor 
		= \left\lfloor\frac{\vp-\theta}{2\pi} \right\rfloor +1
		=\left\lceil\frac{\vp-\theta}{2\pi}\right\rceil.
	\end{equation*}
Thus
	\begin{equation}\label{arg-phi}
	\arg_\varphi (z) = \theta + 2\pi \left\lceil \frac{\varphi - \theta}{2\pi} \right\rceil.
	\end{equation}
To cover the excluded cases when $ \theta = \vp + 2\pi\mathbb{Z} $, we extend \eqref{arg-phi} by writing
	\begin{equation} \label{arg_phi_2}
	\arg_\varphi( z) =  \theta + 2\pi \left\lceil \frac{\varphi - \theta}{2\pi} \right\rceil + 2\pi \chi_{\mathbb{Z}}\left( \frac{\varphi - \theta}{2\pi}\right), 
	\end{equation}
where $ \chi_{\mathbb{Z}}(t) $ is the characteristic function of the set $ \mathbb{Z} $. Formula \eqref{arg_phi_2} holds for any $ \theta $.
 The last term in \eqref{arg_phi_2} involving the characteristic function is a controller, which returns the arguments $ \vp + 2\pi \mathbb{Z} $ to be in the interval $ (\vp, \vp+2\pi] $ as needed. We illustrate this controller term in the following example.
 
\begin{example}
By the $2\pi$-periodicity of the exponential, the argument of $ z=re^{-i\pi} $ can be expressed as the number $(2k+1)\pi$, where $k\in\mathbb{Z}$. Using \eqref{arg_phi_2}, we obtain
		\begin{eqnarray*}
		\arg_{-\pi} (z) &=& (2k+1)\pi + 2\pi \left\lceil \frac{-\pi - (2k+1)\pi}{2\pi} \right\rceil + 2\pi \chi_{\mathbb{Z}}\left( \frac{-\pi - (2k+1)\pi}{2\pi}\right)\\
		&=& (2k+1)\pi+ 2\pi\cdot (-1-k)+2\pi\cdot 1 = \pi,
		\end{eqnarray*}
as desired.
\end{example}

By substituting  \eqref{arg_phi_2} in \eqref{square root}, we obtain a more explicit form for the function $ \s{\cdot}{\vp} $ as
	\begin{equation} \label{sqrt_vfi}
	\s{z}{\vp} =\sqrt{r}\; \exp
			\left\{
				i \left(\frac{\theta}{2} + \pi \mathcal{K}_\vp(\theta) \right)
			\right\}, \quad z=re^{i\theta},
	\end{equation}
where
$$
\mathcal{K}_\vp(\theta) = \left\lceil \frac{\varphi - \theta}{2\pi} \right\rceil +  \chi_{\mathbb{Z}}\left( \frac{\varphi - \theta}{2\pi}\right).
$$

\medskip\noindent
\emph{Proof of Lemma~\ref{different-branches}.}	
By definition, the relation between the square roots $ \s{\cdot}{\vp} $ and $ \s{\cdot}{\psi}  $ is 
	\begin{align*}
	\s{z}{\vp} &= \sqrt{|z|} \exp \left(i \frac{\arg _{\varphi}( z)}{2}\right) \\
	&= \sqrt{|z|} \exp \left(i \frac{\arg _{\psi} (z)}{2}\right) \;  \exp \left(i \frac{\arg _{\varphi} (z)- \arg_{\psi} (z)}{2}\right)\\
	& = \s{z}{\psi}\; \exp\left(i \frac{\arg _{\varphi} (z)- \arg_{\psi} (z)}{2}\right),\quad z\neq 0.
	\end{align*}
By replacing $ \theta $ by $\arg_{\psi} (z)  $ in  \eqref{arg_phi_2}, we obtain
	$$
	 \frac{\arg _{\varphi} (z)- \arg_{\psi} (z)}{2} =  \pi
	 \left(\left\lceil\frac{\varphi-\arg_\psi(z)}{2 \pi}\right\rceil+ \chi_{\mathbb{Z}}\left(\frac{\varphi-\arg_\psi (z)}{2 \pi}\right) \right)
	 = \pi \mathcal{K}_\vp(\arg_\psi (z)).
	$$
Thus
	\begin{equation}\label{sqrt-fi-psi}
	\s{z}{\vp} = \s{z}{\psi} \; \exp\left\{i \pi \mathcal{K}_{\vp}(\arg_\psi (z))\right\}.
	\end{equation}
Since $\mathcal{K}_{\vp}$ is an integer-valued function, the two square roots in \eqref{sqrt-fi-psi}
are either identical or one is $-1$ times the other.\hfill$\Box$

\footnotesize
\bigskip
\noindent
\emph{Gary ~G.~Gundersen}\\
\textsc{University of New Orleans, Department of Mathematics, New Orleans, LA 70148, USA}\\
\texttt{email:ggunders@uno.edu}
\smallskip

\noindent
\emph{Janne~Heittokangas}\\
\textsc{University of Eastern Finland, Department of Physics and Mathematics,\\
P.O. Box 111, 80101 Joensuu, Finland}\\
\texttt{email:janne.heittokangas@uef.fi}
\smallskip

\noindent
\emph{Amine~Zemirni}\\
\textsc{University of Eastern Finland, Department of Physics and Mathematics,\\
P.O. Box 111, 80101 Joensuu, Finland}\\
\texttt{email:amine.zemirni@uef.fi}


\end{document}